\documentclass[a4paper]{amsart}
\usepackage{amsfonts,amsmath,amssymb,amscd,amstext,amsbsy,latexsym,url}
\newtheorem{thm}{Theorem}[section]
\newtheorem{cor}[thm]{Corollary}
\newtheorem{lem}[thm]{Lemma}
\newtheorem{prop}[thm]{Proposition}
\theoremstyle{definition}

\theoremstyle{remark}
\newtheorem{rem}[thm]{Remark}

\numberwithin{equation}{section}

\newcommand{\Max}{\operatorname{Max}}
\newcommand{\RgMax}{\operatorname{RgMax}}
\newcommand{\Maxl}{\operatorname{{Max}_\ell}}
\newcommand{\Maxr}{\operatorname{{Max}_r}}

\newcommand{\Prml}{\operatorname{{Prm}_\ell}}
\newcommand{\Prmr}{\operatorname{{Prm}_r}}
\newcommand{\Char}{\operatorname{Char}}
\newcommand{\lann}{\operatorname{l.ann}}

\newcommand{\End}{\operatorname{End}}

\newcommand{\trdeg}{\operatorname{tr.deg}}

\newcommand{\Aut}{\operatorname{Aut}}

\begin{document}

\title{Noncommutative rings with infinitely many maximal subrings}%
\author{Alborz Azarang}%

\maketitle

\centerline{Department of Mathematics, Faculty of Mathematical Sciences and Computer,}
\centerline{Shahid Chamran University of Ahvaz, Ahvaz-Iran} 
\centerline{azarang@scu.ac.ir}
\centerline{ORCID ID: orcid.org/0000-0001-9598-2411}
\begin{abstract}
We investigate rings with infinitely many or only finitely many maximal subrings. We prove that if $M$ is a maximal left/right ideal of a ring $T$ that is not an ideal of $T$, and $R$ is the idealizer of $M$, then $T$ has at least $|R/M|+1$ maximal left/right ideals that are not two-ideals of $T$; in particular $T$ has at least $|R/M|+1$ distinct maximal subrings. Moreover, if $T$ is a $K$-algebra over an infinite field $K$, then either $T$ has infinitely many maximal subrings or $T$ is a quasi-duo ring with certain algebraic properties similar to those commutative rings that have not infinitely many maximal subrings. We prove that for a simple ring $R$, the ring $R\times R$ has only finitely many maximal subrings if and only if $R$ is finite. We also study rings which are integral over their centers and have only finitely many maximal subrings. We prove that if $T$ is integral over its center and $T$ has more than $2^{\aleph_0}$ maximal (left/right) ideals, then $T$ has infinitely many maximal subrings. In particular, we show that if a $J$-semisimple ring $T$ is integral over its center and has only finitely many maximal subrings, then $T$ embeds into $S\times \prod_{i\in I}E_i$, where each $E_i$ is an absolutely algebraic field and $S$ is a finite semisimple ring. Furthermore, if $T$ is a left Noetherian algebraic $K$-algebra over an infinite field $K$ and $T$ has only finitely many maximal subrings, then $T$ is a countable left Artinian ring which is integral over $\mathbb{Z}_p$, where $p=\Char(K)$. We precisely determine when $T=\prod_{i\in I}\mathbb{M}_{n_i}(E_i)$, where each $E_i$ is a field and $n_i\in\mathbb{N}$, has only finitely many maximal subrings. Finally, we see that if $R$ is an infinite Artinian ring, then $\mathbb{M}_n(R)$, for $n>1$, and $R\times R$ have infinitely many maximal subrings.
\end{abstract}
\section{Introduction}
Let $T$ be an associative ring with identity $1\neq 0$ and let $\RgMax(T)$ denote the set of all  maximal subrings (which contain $1_T$) of $T$. In this paper we study conditions under which $\RgMax(T)$ is finite or infinite. In \cite{azffs}, the author characterized fields and affine rings that have only finitely many maximal subrings; in \cite{azomn}, the authors studied commutative rings with either infinitely many or only finitely many maximal subrings. We recall some relevant facts from those works (which will be needed later) as follows:

\begin{enumerate}
\item A field $E$ has only finitely many maximal subrings if and only if it contains a subfield $F$ without maximal subrings with $[E:F]$ is finite; equivalently, every descending chain $\cdots\subset R_2\subset R_1\subset R_0=E$, where each $R_i$ is a maximal subring of $R_{i-1}$, $i\geq 1$, is finite. Moreover, such a subfield $F$ is unique, all such chains have the same length, say $m$, and $R_m=F$; see \cite[Theorems 3.2 and 3.6]{azffs}.
\item Let $R=F[\alpha_1,\ldots,\alpha_n]$ be an affine integral domain over a field $F$. Then $R$ has only finitely many maximal subrings if and only if $F$ has only finitely many maximal subrings and each $\alpha_i$ is algebraic over $F$; see \cite[Theorem 4.1]{azffs}.
\item For any ring $R$, either $R$ is a Hilbert ring or $\RgMax(R)$ is infinite, see \cite[Corollary 1.9]{azomn}.
\item If $R$ is an Artinian ring that is either uncountable, or of zero characteristic, or is not integral over its prime subring, then $\RgMax(R)$ is infinite; see \cite[Theorem2.5]{azomn}.
\item If $R$ is a Noetherian ring with $|R|>2^{\aleph_0}$, then $|\RgMax(R)|\geq 2^{\aleph_0}$; see \cite[Theorem 2.7]{azomn}.
\item Let $\{R_i\ |\ i\in I\}$ be a family of rings and $R=\prod_{i\in I}R_i$, in \cite[Theorem 3.9]{azomn}, the authors precisely determined when $R$ has only finitely many maximal subrings. Moreover, if $I$ is infinite, then $|\RgMax(R)|\geq 2^{|I|}$; see \cite[Remark 3.18]{azkrc}.
\item For a field $K$, the ring $K\times K$ has only finitely many maximal subrings if and only if $K$ is a finite field; see \cite[Corollary 3.5]{azomn}.
\item If $R$ is a semilocal ring, then either $R$ has infinitely many maximal subrings or $R$ is integral over $\mathbb{Z}_n$, where $\Char(R)=n>1$; see \cite[Corollary 2.2]{azomn}.
\item If $R$ is a ring with $|\Max(R)|>2^{\aleph_0}$ (resp. $|R/J(R)|>2^{2^{\aleph_0}}$), then $|\RgMax(R)|\geq |\Max(R)|$ (resp. $|\RgMax(R)|\geq 2^{\aleph_0}$); see \cite[Proposition2.8]{azomn}.
\item A ring $R$ satisfies that $R\times R$ has only finitely many maximal subrings if and only if $R$ is a semilocal ring with only finitely many maximal subrings and $R/M$ is finite for each $M\in \Max(R)$; see \cite[Corollary 3.10]{azomn}.
\end{enumerate}
In \cite{azomn}, the authors also generalized statement $(1)$ for a finite direct products of fields. For noncommutative rings, the existence of maximal subring was first studied in \cite{azq}; we will mention some results from \cite{azq} in the next section. Interestingly, every ring $R$ can be realized as a maximal subring of some larger ring $T$; see \cite[Theorem 3.7]{azq}. Recently, the existence and the structure of maximal subrings in division rings were investigated in\cite{azdiv}; and the existence of maximal subrings in certain noncommutative rings was studied in \cite{azmscnc}. In \cite{krmnaz} and \cite{azfd}, the authors systematically examined maximal commutative subrings of general noncommutative rings and commutative maximal subrings, respectively. We refer the interested reader to \cite{mscnfsr} for classification of maximal subrings of finite semisimple rings and their applications to calculating the covering number of a finite semisimple ring; see also \cite{azcond, azid} for further results on conductor ideals and certain ideals of the extension $R\subseteq T$, where $R$ is a maximal subring of $T$.\\

In this paper, we extend the study of maximal subring beyond the commutative setting, focusing on general associative rings with identity. Specifically, we investigate conditions under which a ring possesses only finitely many or infinitely many maximal subrings, with particular attention to noncommutative algebras over infinite fields, rings integral over their centers, certain matrix or product constructions. Our approach combines ideal-theoretic techniques, structural ring theory, and cardinality arguments, providing a unified framework that generalizes and complements earlier commutative results.\\  

A brief outline of this paper is as follows. In Section $2$, we present some preliminary results needed in the sequel, particularly concerning the number of maximal left/right ideals of a ring $T$ that are not two-sided ideals. Specifically, we prove that if $T$ is a ring, $M$ is a maximal left/right ideal of $T$ that is not an ideal of $T$, and $R$ is the idealizer of $M$ in $T$, then $T$ has at least $|R/M|+1$ maximal left/right ideals $N$ that are similar to $M$ (i.e., $T/N\cong T/M$ as left/right $T$-modules); consequently, $T$ has at least $|R/M|+1$ maximal left/right ideals that are not ideals of $T$. Moreover, if $T$ is a $K$-algebra over an infinite field $K$, then either $T$ is a quasi-duo ring or $T$ has infinitely many maximal left/right ideals that are not ideals of $T$ (indeed, either $|\Maxl(T)\setminus \Max(T)|\geq |K|$ or $|\Maxr(T)\setminus \Max(T)|\geq |K|$). Finally we recall several required results from \cite{ston} regarding the number of maximal left/right ideals of $\mathbb{M}_n(R)$ (where $R$ is a ring and $n\in\mathbb{N}$), as well as from \cite{azq} concerning maximal subrings of noncommutative rings.\\

In Section $3$, we investigate the number of maximal subrings in general noncommutative rings. First, we show that if $D$ is a division ring with only finitely many maximal subrings, then the center of $D$ is algebraically closed in $D$. For a simple ring $S$, we precisely determine all maximal subrings of $S\times S$. We prove that if $S$ is a simple ring that is not a division ring, then $S$ has only finitely many maximal subrings if and only if $S$ is finite. Consequently, for a simple ring $S$, the ring $S\times S$ has only finitely many maximal subrings if and only if $S$ is a finite ring. Next, we demonstrate that if a ring $T$ has only finitely many maximal subrings, then $\Maxl(T)\setminus \Max(T)$, $\Maxr(T)\setminus \Max(T)$, $\Max(T)\setminus \Maxl(T)$, $\Prml(T)\setminus \Max(T)$ and $\Prml(T)\setminus \Maxl(T)$ (and  their right analogous) are all finite. Furthermore, if $M$ is a maximal left/right ideal of $T$ that is not an ideal of $T$, then $\mathbb{I}(M)/M$ is a finite field; and if $N$ is a maximal ideal of $T$ that is not a maximal left ideal of $T$, then $T/N\cong \mathbb{M}_n(F)$, for some finite field $F$ and some $n>1$. We also show that, if $T$ has only finitely many maximal subrings, then $J(T)=\bigcap_{M\in \Max(T)}M$. In particular, if $T$ is a $K$-algebra over an infinite field $K$, then either $T$ has infinitely many maximal subrings or $T$ is a quasi-duo ring. Finally in Section $3$, we prove that if $T$ is an infinite ring of nonzero characteristic, then either $T$ has infinitely many maximal subrings or the set $\{M\in \Max(T)\ |\ |T/M|<\aleph_0\}$ is finite.\\

Section $4$, is devoted to rings that are integral over their centers and have only finitely many maximal subrings, with special attention to algebraic $K$-algebra over an infinite field $K$. We prove that if $T$ is integral over its center and $T$ has more than $2^{\aleph_0}$ maximal (left/right) ideals, then $T$ has infinitely many maximal subrings. Moreover, if $T$ is an algebraic $K$-algebra over an infinite field $K$, then either $T$ has infinitely many maximal subring or $T/J(T)$ is a commutative von Neumann regular ring that is integral over $\mathbb{Z}_p$, where $p=\Char(K)$. We also establish that if $T$ is integral over its center, then either $T$ has infinitely many maximal subrings or $\Prml(T)=\Max(T)=\Prmr(T)$, and for each maximal ideal $M$ of $T$, either $T/M$ is an absolutely algebraic field (with only finitely many maximal subrings) or $T/M$ is a finite simple ring. In particular, if additionally $J(T)=0$, then $T$ embeds into $S\times\prod_{i\in I}E_i$, where $S$ is a finite semisimple ring and each $E_i$ is an absolutely algebraic field.\\

Finally in Section $5$, we precisely determine when a ring of the form $T=\prod_{i\in I}\mathbb{M}_{n_i}(E_i)$ (with each $E_i$ is a field and $n_i\in\mathbb{N}$), has only finitely many maximal subrings. We also prove that if $R$ is an infinite Artinian ring, then $\mathbb{M}_n(R)$, for $n>1$, as well as $R\times R$, have infinitely many maximal subrings.\\

All rings considered are unital with $1\neq 0$. All subrings, modules and homomorphisms are assumed to be unital. If $R\subsetneq T$ is a ring extension and such that no other subring lies strictly between $R$ and $T$, then $R$ is called a maximal subring of $T$; equivalently, the extension $R\subseteq T$ is called a minimal ring extension, see \cite{dbsid,dbsc,dorsy,frd,abmin}. Let $T$ be a ring and $M$ be a left $T$-module, then $\Maxr(T)$, $\Maxl(T)$, $\Max(T)$ and $\lann_T(M)$, denote the set of all maximal right ideals of $T$, the set of all maximal left ideals of $T$, the set of all maximal ideals of $T$ and the left annihilator of $M$ in $T$, respectively. The set of all left (resp. right) primitive ideals of $T$ is denoted by $\Prml(T)$ (resp. $\Prmr(T)$). The characteristic of a ring $T$ written as $\Char(T)$. We use $\Aut(T)$ for the group of all ring automorphisms of $T$. The Jacobson radical of $T$ is denoted by $J(T)$. We let $N(T)$ and $U(T)$ be the set of all nilpotent elements and the units group of $T$, respectively. The opposite ring of $T$ is written by $T^{\text{op}}$. The ring of all $n\times n$ matrices over $T$ is denoted by $\mathbb{M}_n(T)$. The ring of all $T$-module homomorphisms of $M$ is written as $\End_T(M)$. For a right ideal $A$ of $T$, the idealizer of $A$ in $T$ is denoted by $\mathbb{I}(A)$, it is the largest subring of $T$ in which $A$ is a two-sided ideal, i.e., $\mathbb{I}(A)=\{t\in T\ |\ tA\subseteq A\}$. The same notation is used when $A$ is a left ideal of $T$. In any case, the ring $\mathbb{E}_T(A):=\mathbb{I}(A)/A$ is called the eigenring of $A$. If $X$ is a subset of $T$, then $C_T(X)$ denotes the centralizer of $X$ in $T$; in particular, $C(T)=C_T(T)$ is the center of $T$. The subring $\{n.1_T\ |\ n\in\mathbb{Z}\}$ is called the prime subring of $T$. A ring $T$ is called left (resp. right) quasi-duo if any maximal left (resp. right) ideal of $T$ is two-sided. If $T$ is both left and right quasi-duo, it is called a quasi-duo ring, see \cite{lamq}. Let $R$ be a subring of $T$ and $t\in T$. We say that $t$ is left (resp. right) integral over $R$, if $t$ is a root of a monic polynomial of degree $n\geq 1$ with coefficients on the left (resp. right) side from $R$. The subring $R$ is left (resp. right) integrally closed in $T$, if every element of $T$ that is left (resp. right) integral element over $R$, belongs to $R$. We say it is integrally closed in $T$, if $R$ is both left and right integrally closed in $T$. When $T$ is a ring and $R$ is a subring of $C(T)$, we call $T$ is an $R$-algebra. In this case, the notation of right and left integral elements (as well as, left and right integrally closed) for the extension $R\subseteq T$ coincide. For a field extension $F\subseteq E$, we write $\trdeg(E/F)$ for the transcendence degree of $E$ over $F$. Other notations and definitions follow \cite{good,lam,rvn}.

\section{Preliminaries}
Let $T$ be a ring and $I$ and $J$ be left ideals of $T$. We say $I$ and $J$ are similar left ideals if $T/I\cong T/J$ as left $T$-modules. If $a\in T$, define $(I:a):=\{x\in T\ |\ xa\in I\}$, then it is not hard to see that $I$ and $J$ are similar if and only if there exists $c\in T$ such that $I+Tc=T$ and $J=(I:c)$, see \cite[Proposition 1.3.6]{cohnfir}. For a left ideal $I$ of $T$, we have the ring isomorphism $\mathbb{I}(I)/I\cong (\End_T(T/I))^{\text{op}}$, see \cite[P. 14, 1.11]{macrob}. Consequently, if $I$ and $J$ are similar left ideals of $T$, then $\mathbb{I}(I)/I\cong \mathbb{I}(J)/J$ as rings. Thus, up to ring isomorphism, the ring $\mathbb{I}(I)/I$ is uniquely determined by similarity classes; in particular $|\mathbb{I}(I)/I|=|\mathbb{I}(J)/J|$. In other words, the cardinality of $\mathbb{I}(I)/I$ is invariant under similarity. Now, let $M$ be a maximal left ideal of $T$. The following facts are useful:
\begin{enumerate}
\item A maximal left ideal $N$ of $T$ is similar to $M$ if and only if $N=(M:c)$ for some $c\in T\setminus M$.
\item If $N$ is a (maximal) left ideal of $T$ similar to $M$, then $N$ is an ideal of $T$ if and only if $N=M$, which in turn holds if and only if $M$ is an ideal of $T$.
\end{enumerate}
To see the previous facts, note that $(1)$ follows directly from \cite[Proposition 1.3.6]{cohnfir}, because obviously $M+Tc=T$. For $(2)$, assume that $N$ and $M$ are similar left ideals and that $N$ is an ideal. Since $T/N\cong T/M$ as left $T$-modules, we have $N=\lann_T(T/N)=\lann_T(T/M)\subseteq M$. Because $N$ is a maximal left ideal, the inclusion $N\subseteq M$ forces $N=M$. Hence $(2)$ holds. Consequently, if $M$ is a maximal left ideal of $T$ that is not an ideal of $T$, and if $[M]$ denotes the set of all (maximal) left ideals of $T$ similar to $M$, then by $(2)$ we have $[M]\subseteq \Maxl(T)\setminus \Max(T)$, i.e., every (maximal) left ideal in $[M]$ fail to be a two-ideal. Clearly, similarity is an equivalence relation on the set of all left ideals of $T$, and in particular on $\Maxl(T)$. From $(1)$ and $(2)$, it is evident that for a maximal left ideal $M$, we have $|[M]|=1$ (i.e., $[M]=\{M\}$) if and only if $M$ is an ideal of $T$ (and therefore in this case $M\in \Max(T)\cap \Maxr(T)$).\\

It is clear that for any left ideal $I$ of a ring $T$ and any $c\in T$, we have $(I:c)=T$ if and only if $c\in I$. Moreover, $I\subseteq (I:c)$ holds if and only if $c\in\mathbb{I}(I)$. In particular, if $M$ is a maximal left ideal of $T$, these observations immediately imply that $(M:c)=M$ if and only if $c\in \mathbb{I}(M)\setminus M$. Consequently, $[M]=\{M\}\cup\{(M:c)\ |\ c\in T\setminus\mathbb{I}(M)\}$. Before starting our main result (which present a lower bound for $|[M]|$), we recall that, as noted earlier, for any $N\in [M]$, we have $|\mathbb{I}(M)/M|=|\mathbb{I}(N)/N|$. We now present the following key theorem.

\begin{thm}\label{pt1}
Assume that $T$ is a ring and $M\in \Maxl(T)$ is not an ideal of $T$ (i.e., $M$ is not a maximal ideal of $T$). Then $|[M]|\geq |\mathbb{I}(M)/M|+1$. In particular, $|\Maxl(T)\setminus \Max(T)|\geq |\mathbb{I}(M)/M|+1$.
\end{thm}
\begin{proof}
Let $R=\mathbb{I}(M)$ and $D=R/M$. Since $M$ is not an ideal of $T$ we infer that $R$ is a proper subring of $T$. Take a fixed element $x\in T\setminus R$. We claim that the function $c+M\longmapsto (M:x+c)$ is a one-one function from $D$ to $[M]\setminus \{M\}$. Since $x\in T\setminus R$ and $c\in R$ we deduce that $x+c\in T\setminus R$ and therefore $(M:x+c)\in [M]\setminus \{M\}$, by the previous observation. Hence it remain to show that the function is one-one. Hence assume that $a,b\in R$ and $(M:x+a)=(M:x+b)$, we must show that $a+M=b+M$, i.e., $a-b\in M$. Let $t\in (M:x+a)=(M:x+b)$, then we conclude that $tx+ta, tx+tb\in M$. It follows that $t(a-b)\in M$, i.e., $(M:x+a)\subseteq (M:a-b)$. Now if $a-b\notin M$, then we obtain that $a-b\in R\setminus M$, a fortiori $(M:a-b)=M$. This implies that $(M:x+a)\subseteq M$ and consequently by maximality of $(M:x+a)$ we deduce that $(M:x+a)=M$ which is absurd, because $x+a\in T\setminus R$. Thus $a-b\in M$, as desired.
\end{proof}

An immediate consequence of the previous theorem is that if $T$ is a ring and $\Maxl(T)\setminus \Max(T)\neq\emptyset$, then $|\Maxl(T)\setminus \Max(T)|\geq 3$. In other words, if a ring $T$ possesses a maximal left ideal that is not a two-ideal, then $T$ has at least three distinct maximal left ideals that are not two-ideals.

\begin{cor}\label{pt2}
Let $T$ be a $K$-algebra over a field $K$. Then either $T$ is a left quasi-duo ring or $T$ has at least $|K|+1$ maximal left ideals which are not ideals of $T$ (i.e., $|\Maxl(T)\setminus \Max(T)|\geq |K|+1$).
\end{cor}
\begin{proof}
Assume that $T$ is not a left quasi-duo ring and $M$ is a maximal left ideal of $T$ which is not an ideal of $T$. Since $K$ is a field which is contained in the center of $T$, we conclude that $K\subseteq R:=\mathbb{I}(M)$ and clearly $K\cap M=0$. Thus $K$ embeds in $R/M$ and we are done by Theorem \ref{pt1}.
\end{proof}

Hence we have the following immediate fact.

\begin{cor}\label{pt3}
Let $T$ be a $K$-algebra over an infinite field $K$. Then either $T$ is a left quasi-duo ring or $\Maxl(T)\setminus \Max(T)$ is infinite (in fact, $|\Maxl(T)\setminus \Max(T)|\geq |K|$).
\end{cor}

Thus, if $T$ is a ring and $\Maxl(T)\setminus \Max(T)$ is finite, then for each $M\in \Maxl(T)\setminus \Max(T)$, we deduce that $\mathbb{I}(M)/M$ is a finite field. Moreover, if additionally, $T$ is a $K$-algebra over a field $K$ and $\Maxl(T)\setminus \Max(T)$ is finite, then $K$ must be a finite field (note that for any maximal left ideal $M$ of $T$, the field $\mathbb{I}(M)/M$ contains a copy of $K$).\\

Let $\Maxl(T)\setminus \Max(T)=\bigcup_{i\in I}[M_i]$ be the partition of $\Maxl(T)\setminus \Max(T)$ into similarity classes. By Theorem \ref{pt1}, we obtain that  $|\Maxl(T)\setminus \Max(T)|\geq |I|+\sum_{i\in I}|\mathbb{I}(M)/M|$.\\

It is clear from the above results that if $K$ is a field and $n>1$, then $|\Maxl(\mathbb{M}_n(K))|\geq |K|+1$ and $|\Maxr(\mathbb{M}_n(K))|\geq |K|+1$ (note, $\mathbb{M}_n(K)$ is a simple $K$-algebra and therefore each maximal left/right ideal of $\mathbb{M}_n(K)$ is not an ideal of $\mathbb{M}_n(K)$). In particular, if $K$ is an infinite field, then $\mathbb{M}_n(K)$ has at least $|K|$-many maximal left/right ideals (also see Remark \ref{pt5a}). Next we employ the notation and results of \cite{ston} to show that if $R$ is an infinite division ring and $n>1$, then $T=\mathbb{M}_n(R)$ possesses infinitely many maximal left/right ideals. Recall that if $M$ is a maximal left ideal of a ring $R$ and $T=\mathbb{M}_n(R)$, then for each $u\in R^n\setminus M^n$, the set $D(M,u):=\{X\in T\ |\ Xu\in M^n\}$ is a maximal left ideal of $T$; indeed, $T/D(M,u)\cong (R/M)^n$ as left $T$-modules. Moreover, $\Maxl(T)=\{D(M,u)\ |\ M\in \Maxl(R),\ u\in R^n\setminus M^n\}$, see \cite[Theorem 1.2]{ston}. In particular, when $R$ is a division ring, we obtain that $\Maxl(\mathbb{M}_n(R))=\{D(0,u)\ |\ 0\neq u\in R^n\}$.\\

Hence to show that $\Maxl(\mathbb{M}_n(R))$ is infinite for an infinite division ring $R$, it suffices to produce infinitely many nonzero vectors $u\in R^n$ for which $D(0,u)$ are distinct. By \cite[Proposition 2.3]{ston}, if $M$ is a maximal left ideal of a ring $R$, $S=\mathbb{I}(M)$ its idealizer of $M$ in $R$, and $u,v\in S^n$, then $D(M,u)=D(M,v)$ (in $T=\mathbb{M}_n(R)$) if and only if there exists $c\in S\setminus M$ such that $v\equiv uc\ \mod(M)$, i.e., $v-uc\in M^n$. In particular, for a division ring $R$ the only possible maximal left ideals is $M=0$, so that $S=R$; consequently, for $u,v\in R^n$, we have $D(0,u)=D(0,v)$ if and only if there exists $0\neq c\in R$ such that $v=uc$. In other words $v$ and $u$ are right-parallel vectors in the right $R$-module $R^n$. Clearly, if $R$ is an infinite division ring and $n>1$, then as right $R$-module $R^n$ contains infinitely many right non-parallel vectors. This observations yield the following quick corollary.

\begin{cor}\label{pt4}
Let $R$ be an infinite division ring and $n>1$. Then $T=\mathbb{M}_n(R)$ has infinitely many maximal left (resp. right) ideals.
\end{cor}

\begin{rem}\label{pt5a}
Let $R$ be a ring, $n>1$, and $T=\mathbb{M}_n(R)$. It is obvious that $\Maxl(T)\cap \Max(T)=\emptyset=\Maxr(T)\cap \Max(T)$. Moreover, the map $M\longmapsto M^t$ (where $M^t:=\{m^t\ |\ m\in M\}$ is the transpose of $M$), is a one-to-one correspondence from $\Maxl(T)$ onto $\Maxr(T)$. In particular, $|\Maxl(T)|=|\Maxr(T)|$.
\end{rem}

\begin{rem}\label{pt5}
For any ring $T$, we have the following immediate facts:
\begin{enumerate}
\item $\Maxl(T)\cap \Max(T)=\Maxr(T)\cap \Max(T)=\Prml(T)\cap \Maxl(T)=\Prmr(T)\cap \Maxr(T)=\Prml(T)\cap \Maxr(T)=\Prmr(T)\cap \Maxl(T)$.
\item $\Max(T)\setminus \Maxl(T)=\Max(T)\setminus \Maxr(T)$.
\end{enumerate}
\end{rem}

Next, let we recall some basic facts from maximal subrings, as well as several needed results from \cite{azq}. A proper subring $S$ of a ring $T$ is maximal if and only if for every $x\in T\setminus S$, one has $S[x]=T$. Moreover, if $R$ is a proper subring of $T$ and there exists $\alpha\in T$ such that $R[\alpha]=T$, then by a standard application of Zorn's Lemma, $T$ possesses a maximal subring $S$ with $R\subseteq S$ and $\alpha\notin S$. In particular, if $R\subsetneq T$ is a ring extension which is finite as ring (or as module) over $R$, then $T$ has a maximal subring containing $R$. Now let $R$ be a ring and $\Delta:=\{(r,r)\ |\ r\in R\}$, be the diagonal subring of $T:=R\times R$. It is easy to see that $\Delta[(1,0)]=T$. Consequently, $T$ has a maximal subring containing $\Delta$. In fact, every subring $S$ of $T$ that contains $\Delta$ has the form $\Delta(I):=\Delta+I\times I$, where $I$ is an ideal of $R$. Moreover, $\Delta(I)$ is a maximal subring of $T$ if and only if $I$ is a maximal ideal of $R$, see \cite[Theorem 2.7]{azq}. Therefore a ring $R$ is simple exactly when $\Delta$ is a maximal subring of $R\times R$. The following results from \cite{azq} will be useful.

\begin{thm}\label{pt6}
\begin{enumerate}
\item \cite[Theorem 4.1]{azq}. Let $T$ be a ring and $A$ be a maximal right/left ideal of $T$ that is not an ideal of $T$. Then the idealizer of $A$ is a maximal subring of $T$. In particular, every ring either has a maximal subring or is a quasi-duo ring.
\item \cite[Proposition 4.2]{azq}. Let $R$ be a maximal subring of a ring $T$ that contains a maximal left/right ideal $A$ of $T$ that is not two-sided ideal of $T$. Then $R$ is precisely the idealizer of $A$ in $T$.
\item \cite[Theorem 4.4]{azq}. Let $T$ be a ring that is not a division ring. If $T$ is left/right primitive ring (in particular, if $T$ is simple), then $T$ has a maximal subring.
\item \cite[Theorem 3.1]{azq}. For every ring $R$ and every $n>1$, the matrix ring $\mathbb{M}_n(R)$ has a maximal subring.
\end{enumerate}
\end{thm}

Finally, we prove several auxiliary facts. Recall that if $T$ is a ring with $|T|<\mathfrak{a}$, where $\mathfrak{a}$ is an infinite cardinal number, then for each finitely generated left/right $T$-module $M$, we have $|M|<\mathfrak{a}$. Indeed, if $M$ is finitely generated, there exists a natural number $n$ such that $M$ is a homomorphic image of $T^n$ (as left/right $T$-module); hence $|M|\leq |T|^n<\mathfrak{a}^n=\mathfrak{a}$.

\begin{lem}\label{pt7}
Let $T$ be a left Noetherian ring and $I$ a nilpotent ideal of $T$. If $\mathfrak{a}$ is an infinite cardinal number, then $|T|<\mathfrak{a}$ if and only if $|T/I|<\mathfrak{a}$.
\end{lem}
\begin{proof}
The implication $|T|<\mathfrak{a} \ \Longrightarrow |T/I|<\mathfrak{a}$ is obvious for every ideal $I$ of $T$. Conversely, assume $I$ is a nilpotent ideal of $T$ and $|T/I|<\mathfrak{a}$. Since $I$ is nilpotent, there exists a natural number $n$ with $I^m=0$. Consider the chain $0=I^m\subseteq I^{m-1}\subseteq\cdots\subseteq I\subseteq T$. Because $T$ is left Noetherian, each $I^i$ is a finitely generated left ideal. Consequently, for each $i$, $0\leq i\leq m-1$, the left $T$-module $I^i/I^{i+1}$ is finitely generated, and therefore $|I^i/I^{i+1}|<\mathfrak{a}$. In particular, $|I^{m-1}|<\mathfrak{a}$. By descending induction we obtain $|I^i|<\mathfrak{a}$, for all $i$, and finally $|T|<\mathfrak{a}$.
\end{proof}

Note that in the proof of the previous lemma we only used that $I$ is a nilpotent ideal whose powers $I^{i}$ are finitely generated left ideals of $T$. Under the hypotheses of the lemma, we immediately conclude that $T$ is finite (resp. countable) precisely when $T/I$ is finite (resp. countable).\\

We close this section with a result concerning the structure of maximal subrings of certain direct products of rings. Two rings $A$ and $B$ are called homomorphically non-isomorphic, if for every proper ideal $I$ of $A$ and every proper ideal $J$ of $B$, the rings $A/I$ and $B/J$ are not isomorphic. For example, if $A=\mathbb{M}_{n_1}(R_1)\times\cdots\times \mathbb{M}_{n_k}(R_k)$, where $k\geq 1$, each $n_i\geq 2$ and $R_i$ are arbitrary rings and $B$ is any commutative ring, then $A$ and $B$ are homomorphically non-isomorphic (since every homomorphic image of $A$ is noncommutative). With this preparation we state the following.

\begin{lem}\label{pt8}
Let $A$ and $B$ be homomorphically non-isomorphic rings and $T=A\times B$. Then $R$ is a maximal subring of $T$ if and only if either $R=A_1\times B$, where $A_1$ is a maximal subring of $A$, or $R=A\times B_1$, where $B_1$ is a maximal subring of $B$.
\end{lem}
\begin{proof}
It is clear that if any subring of the forms described above is maximal in $T$. Conversely, let $R$ be a maximal subrings of $T$. If $0\times B\subseteq R$, then it is not hard to see that $R=A_1\times B$, where $A_1$ is a maximal subring of $A$. Similarly, if $A\times 0\subseteq R$, then $R=A\times B_1$, where $B_1$ is a maximal subring of $B$. Hence assume that $R$ contains neither $A\times 0$ nor $0\times B$. Let $M:=(A\times 0)\cap R=\{(a,0)\ |\ a\in A_1\}=A_1\times 0$, where $A_1$ is a proper subset of $A$ and $N:=(0\times B)\cap R=\{(0,b)\ |\ b\in B_1\}=0\times B_1$, where $B_1$ is a proper subset of $B$. Clearly $M$ and $N$ are ideals of $R$ and $M+N=A_1\times B_1$. Also note that since $A\times 0\nsubseteq R$ and $R$ is a maximal subring of $T$, we have $R+(A\times 0)=T$; Consequently, $R/M\cong B$. Similarly, $R/N\cong A$. Now we distinguish two cases (recall that $(1,1)\in R$): $(\bf i)$ $(1,0)\in R$ (equivalently $(0,1)\in R$). Then $M+N=R$, so $R=A_1\times B_1$. Hence $A_1$ and $B_1$ are proper subrings of $A$ and $B$, respectively. But then $R=A_1\times B_1\subsetneq A\times B_1\subsetneq T$, contradicting maximality of $R$. $({\bf ii})$ $(1,0)\notin R$ and $(0,1)\notin R$. Thus $(1,1)\notin M+N$, so $M+N$ is a proper ideal of $R$ contains $M$ and $N$. From $R/M\cong B$ and $R/N\cong A$, we obtain that $A$ and $B$ have isomorphic homomorphic images, which is impossible, because $A$ and $B$ are homomorphically non-isomorphic. Since both cases lead to a contradiction, our assumption that $R$ contains neither $A\times 0$ nor $0\times B$ cannot hold. Thus $R$ must be one of the two forms stated in the lemma.
\end{proof}

\section{Rings with Infinitely Many Maximal Subrings}
In this section we study the number of maximal subrings for general noncommutative rings, with particular emphasises on division rings, simple rings and $K$-algebras over an infinite field $K$. Let $E$ be a field with prime subfield $F$. By \cite[Corollary 1.5]{azomn}, if $E$ has only finitely many maximal subrings, then $F\cong\mathbb{Z}_p$ for some prime number $p$ and $E$ is algebraic over $F$ (for a full characterization of such fields see \cite{azffs}). Moreover, if $E$ is not an absolutely algebraic field, then $|\RgMax(E)|
\geq \max\{\trdeg(E/F), |E|\}$, by \cite[Corollary 1.5]{azomn}. For division rings we obtain the following result.

\begin{prop}\label{t1}
Let $D$ be a noncommutative division ring. If $D$ has a maximal subring $R$ which is a division ring, then $D$ has infinitely many maximal subrings. Consequently, if $D$ has only finitely many maximal subrings, then each $x\in D\setminus C_D(D)$ is not algebraic over $C_D(D)$ (i.e., $C_D(D)$ is algebraically closed in $D$) and $C_D(R)=C_D(D)$ for each maximal subring $R$ of $D$.
\end{prop}
\begin{proof}
Assume that $R$ is a maximal subring of $D$ which is a division ring. Then by \cite[Theorem 3]{faithcd}, there exist infinitely many subrings of the form $xRx^{-1}$, where $x\in D\setminus\{0\}$. It is evident that for each $0\neq x\in D$, $xRx^{-1}$ is a maximal subring of $D$. Thus $D$ has infinitely many maximal subrings. For the next part, assume that $D$ has only finitely many maximal subrings, then by the first part we conclude that each maximal subring of $D$ is not a division ring. Now, if $x$ is a non-centeral element of $D$ which is algebraic over the center of $D$, then by \cite[Corollary 2.10]{azdiv}, $D$ has a maximal subring which is a division ring which is impossible. Hence the center of $D$ is algebraically closed in $D$. For the final part note that if $R$ is a maximal subring of $D$ and $\alpha\in C_D(R)$ is not a central element of $D$, then clearly $R=C_D(\alpha)$, i.e., $R$ is a division ring which is absurd. Thus $C_D(R)=C_D(D)$.
\end{proof}

\begin{cor}\label{t1a}
Let $D$ be an infinite division ring with $C_D(D)=F$. Then either $D$ has infinitely many maximal subrings or $[D:F]\geq |F|$ (in fact, for each $x\in D\setminus F$ we have $[F(x):F]\geq |F|$).
\end{cor}
\begin{proof}
Let $x\in D\setminus F$ and $[F(x):F]<|F|$. Then analogous to the proof of \cite[Theorem 4.20]{lam}, we deduce that $F(x)$ is algebraic over $F$. Hence $x$ is algebraic over $F$ and thus we are done by the previous result.
\end{proof}

\begin{cor}\label{t2}
Let $D$ be an existentially complete noncommutative division ring over a field $K$. Then $D$ has infinitely many maximal subrings.
\end{cor}
\begin{proof}
By \cite[$(1)$ of Proposition 3.5]{azdiv}, $D$ has a maximal subring of the form $R=C_D(x)$ for some $x\in D$. It is clear that $R$ is a division ring and hence we are done by Proposition \ref{t1}.
\end{proof}

To obtain further results, we first need to characterize the maximal subrings of $R\times R$ where $R$ is a simple ring. This done as follows. 

\begin{thm}\label{t3}
Let $R$ be a simple ring. Then $S$ is a maximal subring of $T:=R\times R$ if and only if $S$ satisfies exactly in one of the following conditions:
\begin{enumerate}
\item $S=A\times R$ or $S=R\times A$, where $A$ is a maximal subring of $R$.
\item There exist $\sigma_1, \sigma_2\in \Aut(R)$ such that $S=\{(\sigma_1(x),\sigma_2(x))\ |\ x\in R\}$.
\end{enumerate}
Moreover in case $(2)$, we have $S\cong R$.
\end{thm}
\begin{proof}
First assume that $S$ is a subring of $T$ which is of the forms given in $(1)$ or $(2)$. It is clear that if $S$ is of the forms given in $(1)$, then $S$ is a maximal subring of $T$. Hence assume that $S$ is of the form in $(2)$. Define $\sigma: R\longrightarrow S$, by $\sigma(x)=(\sigma_1(x),\sigma_2(x))$, then one can easily check that $\sigma$ is an isomorphism and therefore $S\cong R$ is a simple ring. To show $S$ is a maximal subring of $T$, we first prove that $S[(1,0)]=T$. Clearly $(0,1)\in S[(1,0)]$, because $(1,1)\in S$. Now let $(a,b)\in T=R\times R$, thus there exist $x,y\in R$ such that $\sigma_1(x)=a$ and $\sigma_2(y)=b$, because $\sigma_1, \sigma_2\in \Aut(R)$. Thus by definition of $S$ we conclude that $(\sigma_1(x),\sigma_2(x))=(a,\sigma_2(x))\in S$ and $(\sigma_1(y),\sigma_2(y))=(\sigma_1(y),b)\in S$. Therefore $(a,b)=(a,\sigma_2(x))(1,0)+(\sigma_1(y),b)(0,1)\in S[(1,0)]$. Thus $S[(1,0)]=T$. Now we prove that for each $(x,y)\in T\setminus S$, we have $S[(x,y)]=T$, which shows that $S$ is a maximal subring of $T$. Since $\sigma_2\in \Aut(R)$, there exists $c\in R$ such that $\sigma_2(c)=y$. This implies that $(\sigma_1(c),y)\in S\subseteq S[(x,y)]$ and a fortiori $(x-\sigma_1(c),0)\in S[(x,y)]$. Since $(x,y)\notin S$, we infer that $x_0:=x-\sigma_1(c)\neq 0$. Thus $Rx_0R=R$, because $R$ is a simple ring. Hence there exist $n\in\mathbb{N}$ and $r_i,s_i\in R$, $1\leq i\leq n$, such that $\sum_{i=1}^n r_ix_0s_i=1$. From $\sigma_1\in \Aut(R)$, we deduce that there exist $a_i, b_i\in R$, $1\leq i\leq n$, such that $\sigma_1(a_i)=r_i$ and $\sigma_1(b_i)=s_i$. It follows that
$$(1,0)=\sum_{i=1}^n (\sigma_1(a_i),\sigma_2(a_i))(x_0,0)(\sigma_1(b_i),\sigma_2(b_i))\in S[(x,y)]$$
This implies that $T=S[(1,0)]\subseteq S[(x,y)]$ and then $S[(x,y)]=T$. Consequently, $S$ is a maximal subring of $T$.\\
Conversely, assume that $S$ is a maximal subring of $T$. Suppose $I=R\times 0$ and $J=0\times R$. If either $I\subseteq S$ or $J\subseteq S$, then $S$ is of the forms which is given in $(1)$. Thus assume that $I$ and $J$ are not contained in $S$. Therefore $I_1=I\cap S=\{(x,o)\ |\ x\in S_1\}=S_1\times 0$ and $J_1=J\cap S=\{(0,y)\ |\ y\in S_2\}=0\times S_2$, where $S_1$ and $S_2$ are proper subsets of $R$. Clearly $I_1+J_1=S_1\times S_2\subseteq S$. Also note that since $S$ is a maximal subring of $T$ which does not contain $I$ and $J$, we deduce that $S+I=T=S+J$ and thus $S/I_1\cong R\cong S/J_1$. Consequently, $I_1$ and $J_1$ are maximal ideals of $S$. Now we have two cases: $(\bf a)$ If $(1,0)\in S$ (which is equivalent to $(0,1)\in S$), then $(1,1)=(1,0)+(0,1)\in I_1+J_1$, which means $I_1+J_1=S$ and therefore $S=S_1\times S_2$. Thus $S_1$ and $S_2$ are proper subring of $R$. But obviously $S\subsetneq R\times S_2\neq T$, which contradicts to the maximality of $S$. $(\bf b)$ Hence we may suppose that $(1,0), (1,0)\notin S$. It implies that $I_1+J_1=S_1\times S_2$ is a proper ideal of $S$ which contains $I_1$ and $J_1$. But as we see earlier, $I_1$ and $J_1$ are maximal ideals of $S$. Therefore we obtain that $I_1=J_1$. So $I_1=J_1=0$ and then $S\cong R$. Let $\sigma: R\longrightarrow S$ be a ring isomorphism, and for each $x\in R$, let $\sigma(x)=(\sigma_1(x),\sigma_2(x))$. Then clearly $\sigma_1$ and $\sigma_2$ are ring endomorphisms of $R$. Since $R$ is a simple ring we deduce that $\sigma_1$ and $\sigma_2$ are one-one. Finally note that from $S=\{(\sigma_1(x),\sigma_2(x))\ |\ x\in R\}$ and $(1,0)\notin S$, we have $S[(1,0)]=T$. It follows that $T=S+S(1,0)$ (note, $(1,0)$ is a central idempotent of $T$). To show that $\sigma_1$ and $\sigma_2$ are onto, let $x\in R$, then $(x,0)\in S+S(1,0)$. Consequently there exist $a,b\in R$ such that $(x,0)=(\sigma_1(a),\sigma_2(a))+(\sigma_1(b),\sigma_2(b))(1,0)$, which shows that $\sigma_1(b)=x$, i.e., $\sigma_1$ is onto. Analogously, $\sigma_2$ is onto, and hence we are done.
\end{proof}

We now present a main result of this paper concerning simple rings.

\begin{thm}\label{t4}
Let $T$ be a simple ring which is not a division ring. Then $T$ has only finitely many maximal subring if and only if $T$ is finite (i.e., $T\cong\mathbb{M}_n(F)$, where $F$ is a finite field and $n\geq 2$)
\end{thm}
\begin{proof}
We first show that $\Maxl(T)$ is finite. Otherwise, let $M_1, M_2,\ldots$ be distinct maximal left ideals of $T$. Then for each $i$, by Theorem \ref{pt6}$(1)$, $A_i:=\mathbb{I}_T(M_i)$ is a maximal subring of $T$ which contains $M_i$ (note, for each $i$, $M_i$ is not an ideal of $T$, because $T$ is simple ring which is not a division ring). Since whenever $i\neq j$, we have $M_i+M_j=T$, then we have $A_i\neq A_j$. Thus $T$ has infinitely many maximal subrings which is absurd. Thus $\Maxl(T)=\{M_1,\ldots, M_n\}$. Now from the fact that $J(T)=0$, we deduce that $M_1\cap\cdots\cap M_n=0$. Consequently, $T$ embeds in $T/M_1\times\cdots\times T/M_n$. A fortiori, $T$ is a left Artinian ring. It follows that $T\cong\mathbb{M}_n(D)$, for some division ring $D$ and $n\geq 2$ (note, $T$ is not a division ring). If $D$ is an infinite division ring, then $\mathbb{M}_n(D)$, where $n\geq 2$, has infinitely many maximal left (right) ideals, by Corollary \ref{pt4}. Thus $D$ is finite and then $D$ is a finite field. This implies that $T$ is a finite ring, which complete the proof.
\end{proof}

We recall that for a field $K$, the ring $K\times K$ has only finitely many maximal subrings if and only if $K$ is finite, see \cite[Corollary 3.5]{azomn}. The following result generalizes this fact.

\begin{prop}\label{t5}
Let $R$ be a simple ring. Then $R\times R$ has only finitely many maximal subrings if and only if $R$ is finite.
\end{prop}
\begin{proof}
It is trivial that if $R$ is finite then $R\times R$ has only finitely many maximal subrings. Conversely, assume that $R\times R$ has only finitely many maximal subrings. Theorem \ref{t3} implies that $R$ has only finitely many maximal subrings. If $R$ is not a division ring, then we are done by the previous theorem. Thus suppose that $R$ is a division ring and $R\times R$ has only finitely many maximal subrings. By the preceding comment, we may assume that $R$ is not a field, i.e., $C_R(R)\subsetneq R$. Now, let $R$ be an infinite ring. Then we have two cases: $(\bf a)$ If there exists a non-central element $x$ of $R$ which is algebraic over the center of $R$, then by Proposition \ref{t1}, $R$ has infinitely many maximal subrings, which is impossible. $(\bf b)$ Hence assume that each non-central element $x$ of $R$ is not algebraic over the center of $R$. Let $x\in R$ be a non-central element and for each $n\geq 1$ and $a\in D$, let $\sigma_n(a)=x^n a x^{-n}$ be the inner automorphism of $R$ induced from $x^n$. If $\sigma_n=\sigma_m$ for some $n>m$, then $\sigma_1^{n-m}=1$ and therefore $x^{n-m}$ is a central element of $R$. This implies that $x$ is integral over the center of $R$, which contradicts to the assumption that $x$ is not algebraic over the center. Thus $\sigma_n\neq \sigma_m$ for each $n\neq m$. Consequently, by Theorem \ref{t3}$(2)$, for each $n\geq 1$, the set $S_n=\{(x,\sigma_n(x))\ |\ x\in R\}$ is a maximal subring of $R\times R$. Obviously, whenever $n\neq m$, we have $S_n\neq S_m$. It follows that $R\times R$ has infinitely many maximal subrings, which is absurd. Thus $R$ is a finite, as desired.
\end{proof}

By the previous results we have the following immediate results.

\begin{cor}\label{t6}
Let $D$ be a division ring. Then $D\times D$ has only finitely many maximal subrings if and only if $D$ is a finite field.
\end{cor}

For definition of existentially complete division rings see \cite{cohndr}.

\begin{prop}\label{t7}
Let $D$ be an infinite countable existentially complete division ring, then $D\times D$ has at least $2^{\aleph_0}$ maximal subrings.
\end{prop}
\begin{proof}
By \cite[Theorem 6.5.11]{cohndr}, $D$ has at least $2^{\aleph_0}$ automorphisms. Hence by Theorem \ref{t3} for each $\sigma\in \Aut(D)$, $S_{\sigma}:=\{(x,\sigma(x))\ |\ x\in D\}$ is a maximal subring of $D\times D$. It is clear that $S_{\sigma}\neq S_{\sigma'}$, whenever $\sigma$ and $\sigma'$ are two distinct automorphisms of $D$. Hence $D\times D$ has at least $2^{\aleph_0}$ maximal subrings.
\end{proof}

\begin{cor}\label{t7a}
Let $T=R_1\times R_2$, where $R_1$ and $R_2$ be two simple ring. Then $T$ has infinitely many maximal subrings if and only if one of the following holds:
\begin{enumerate}
\item $R_1$ or $R_2$ has infinitely many maximal subrings.
\item $R_1\cong R_2$ are infinite rings.
\end{enumerate}
\end{cor}
\begin{proof}
It is clear that if $(1)$ holds then $T$ has infinitely many maximal subrings. Now assume that $(2)$ holds, thus $T\cong R_1\times R_1$ and therefore by Proposition \ref{t5}, we conclude that $T$ has infinitely many maximal subrings. Conversely, assume that $T$ has infinitely many maximal subrings, we prove that $(1)$ or $(2)$ holds. Suppose that $(2)$ does not hold, i.e., $R_1$ and $R_2$ are not isomorphic as rings. Thus $R_1$ and $R_2$ are homomorphically non-isomorphic rings. Consequently, by Lemma \ref{pt8}, we infer that each maximal subrings of $T$ is of the form $S_1\times R_2$ or $R_1\times S_2$, where $S_i$ is a maximal subring of $R_i$. Since $T$ has infinitely many maximal subrings, we obtain that either $R_1$ or $R_2$ has infinitely many maximal subrings and hence we are done.
\end{proof}

For a ring $T$, let us denote the sets $\Maxl(T)\setminus \Max(T)$ by $\Maxl'(T)$, similarly $\Maxr'(T):=\Maxr(T)\setminus \Max(T)$ and $\Max'(T):=\Max(T)\setminus \Maxl(T)=\Max(T)\setminus \Maxr(T)$. Note that if $P$ is a prime ideal of a ring $T$ and $I$ be a left (resp. right) ideal and $J$ is an ideal of $T$ such that $P=I\cap J$, then either $P=I$ or $P=J$. To see this, note that  $JI\subseteq I\cap J\subseteq P$ (resp. $IJ\subseteq I\cap J\subseteq P$), therefore $I\subseteq P$ or $J\subseteq P$. Hence $P=I$ or $P=J$.

\begin{prop}\label{t8}
Let $T$ be a ring with only finitely many maximal subrings. Then the following hold:
\begin{enumerate}
\item $\Maxl'(T)$ and $\Maxr'(T)$ are finite sets and in fact $|\Maxl'(T)|\leq |\RgMax(T)|$ and $|\Maxr'(T)|\leq |\RgMax(T)|$.
\item For each $M\in \Maxl'(T) \cup \Maxr'(T)$, the ring $\mathbb{I}(M)/M$ is a finite field. In particular, $C(T)/(C(T)\cap M)$ is a finite field.
\item For each $M\in \Maxl(T)\cap \Max(T)=\Maxr(T)\cap \Max(T)$, the center of $T/M$ is algebraically closed in $T/M$.
\item For each $M\in \Max'(T)$, there exist a natural number $n>1$ and a finite field $F$ such that $T/M\cong\mathbb{M}_n(F)$.
\item $\Max'(T)$ is finite and in fact $|\Max'(T)|\leq |\RgMax(T)|$.
\item $\Prml(T)\setminus \Maxl(T)$, $\Prml(T)\setminus \Max(T)$, $\Prmr(T)\setminus \Maxr(T)$ and $\Prmr(T)\setminus \Max(T)$ are finite sets.
\item If $P$ a left (resp. right) primitive ideal of $T$, then either $P$ is a maximal ideal of $T$ or $P=\bigcap_{i\in I} M_i$, where $M_i\in \Maxl(T)\cap \Max(T)$, moreover in this case $I$ is infinite and $T/P$ is a domain (i.e., $P$ is a completely prime ideal of $T$). In particular, $T/P$ is embeddable in a division ring.
\item $J(T)=\bigcap_{M\in \Max(T)} M$.
\end{enumerate}
\end{prop}
\begin{proof}
$(\bf 1)$ For each maximal left ideal $A$ of $T$ which is not an ideal of $T$, by Theorem \ref{pt6}, $\mathbb{I}(A)$ is a maximal subring of $T$ which contains $A$. It is clear that whenever $A\neq B$ are in $\Maxl'(T)$, then $A+B=T$ and consequently $\mathbb{I}(A)\neq\mathbb{I}(B)$. Therefore the function $A\longmapsto \mathbb{I}(A)$, from $\Maxl'(T)$ into $\RgMax(T)$ is one-one, which complete the proof of $(1)$ for $\Maxl(T)$. The proof for $\Maxr'(T)$ is analogous.\\
$(\bf 2)$ The first part is an immediate consequence of Theorem \ref{pt1} and $(1)$ (note that if $M\in \Maxl'(T)$, then $\mathbb{I}(M)/M\cong (\End({}_T(T/M)))^{op}$ is a division ring). For the last part of $(2)$, obviously $C(T)+M$ is a subring of $T$ which is contained in $\mathbb{I}(M)$. It follows that $(C(T)+M)/M$ is a subring of $\mathbb{I}(M)/M$. Now by the first part of $(2)$, we deduce that $(C(T)+M)/M$ is a finite field. A fortiori, $C(T)/(C(T)\cap M)$ is a finite field.\\
$(\bf 3)$ If $T/M$ is a field (in particular, if $T/M$ is finite), then the conclusion is obvious. Hence assume that $T/M$ is an infinite division ring. By Theorem \ref{t1}, we deduce that the center of $T/M$ is algebraically closed in $T/M$, because otherwise, $T/M$ and therefore $T$ have infinitely many maximal subrings, which is absurd.\\
$(\bf 4)$ This is an immediate consequence of Theorem \ref{t4} and the fact that $T/M$ has only finitely many maximal subrings (note that $T$ has only finitely many maximal subrings and thus $T/M$ has only finitely many maximal subrings).\\
$(\bf 5)$ Suppose that $\Max'(T)=\{M_i\ |\ i\in I\}$, where $M_i\neq M_j$ for $i\neq j\in I$. Then by $(4)$, for each $i\in I$ there exist a natural number $n_i>1$ and a finite field $F_i$ such that $T/M_i\cong\mathbb{M}_{n_i}(F_i)$. Therefore by Theorem \ref{pt6}$(4)$, we infer that for each $i\in I$, $M_i$ is contained in a maximal subring $S_i$ of $T$. Since for $i\neq j$, we have $M_i+M_j=T$, we deduce that $S_i\neq S_j$. Therefore $|I|\leq |\RgMax(T)|$ and hence we are done.\\
$(\bf 6)$ Let $\{P_i\ |\ i\in I \}$ be a family of distinct left primitive ideals of $T$ such that for any $i\in I$, $P_i$ is not a maximal (left) ideal of $T$. Since for each $i\in I$, $P_i$ is a left primitive ideal of $T$, we deduce that there exists a maximal left ideal $M_i$ of $T$ such that $P_i=\lann_T(T/M_i)$. It is clear that whenever $i\neq j$, then $M_i\neq M_j$, for $P_i\neq P_j$. Also note that if $M_i$ is an ideal of $T$, then $P_i=M_i$ is a maximal (left) ideal of $T$ which is absurd. Thus for each $i\in I$ we obtain that $M_i\in \Maxl'(T)$ and therefore by $(1)$ we conclude that $I$ is finite. It follows that $\Prml(T)\setminus \Maxl(T)$, $\Prml(T)\setminus \Max(T)$ are finite. The proof for $\Prmr(T)\setminus \Maxr(T)$ and $\Prmr(T)\setminus \Max(T)$ are analogous.\\
$(\bf 7)$ Let $P$ be a left primitive ideal of $T$. Hence $P$ is an intersection of a family of maximal left ideals of $T$. Since by $(1)$, $\Maxl'(T)$ is finite, we conclude that there exist $N_1,\ldots,N_k\in \Maxl'(T)$, $k\geq 0$, and a family $\{M_i\ |\ i\in I\}\subseteq \Maxl(T)\cap \Max(T)$ such that $P=N_1\cap\cdots N_k\cap N$, where $N=\bigcap_{i\in I}M_i$. By the comments preceding the proposition we infer that either $P=N_1\cap\cdots N_k$ or $P=\bigcap_{i\in I}M_i$. If $P=N_1\cap\cdots N_k$, then the ring $T/P$ embeds in the semisimple left $T$-module $T/N_1\times \cdots\times T/N_k$, as left $T$-modules. This implies that $T/P$ is a prime semisimple ring. Consequently $T/P$ is a simple ring, i.e., $P$ is a maximal ideal of $T$. Hence we may suppose that $P$ is not a maximal ideal of $T$ and $P=N=\bigcap_{i\in I}M_i$. Since $P$ is not maximal ideal of $T$, but is a prime ideal of $T$, we conclude that $I$ is infinite. Finally note that for each $i\in I$, the ring $T/M_i$ is a division ring. A fortiori, $\prod_{i\in I}T/M_i$ is a reduced ring. Since $T/P$ embeds in $\prod_{i\in I}T/M_i$, we deduce that $T/P$ is a reduced ring. Thus $T/P$ is a reduced prime ring and consequently $T/P$ is a domain. It follows by \cite[Theorem 1.2.4]{cohndr} that $T/P$ is embeddable in a division ring.\\
$(\bf 8)$ This is trivial by $(7)$ and the fact that $J(T)$ is the intersection of all left (resp. right) primitive ideals of $T$.
\end{proof}

The previous observations yield the following quick corollaries.

\begin{cor}
Let $T$ be an infinite left/right primitive ring which is not embeddable in a division ring, then $T$ has infinitely many maximal subrings.
\end{cor}
\begin{proof}
If $T$ is a simple ring then we are done by Theorem \ref{t4}. Hence assume that $T$ is not a simple ring. If $T$ has only finitely many maximal subrings, then by Proposition \ref{t8}$(7)$, we conclude that $T$ is embeddable in a division ring which is absurd. Thus $T$ has infinitely many maximal subrings.
\end{proof}

\begin{cor}
Let $T$ be an infinite prime ring which is not embeddable in a division ring. If $J(T)=0$, then $T$ has infinitely many maximal subrings.
\end{cor}
\begin{proof}
Assume that $T$ has only finitely many maximal subrings, then by $(8)$ and $(5)$ of Proposition \ref{t8}, we conclude that $0=J(T)=Q_1\cap\cdots\cap Q_n\cap N$, where $Q_i$ are maximal ideals of $T$ which are not maximal left ideals, $n\geq 0$ and $N=\bigcap_{i\in I} M_i$, where $\Max(T)\cap \Maxl(T)=\{M_i\ |\ i\in I\}$. Since $T$ is a prime ring, we infer that either there exists $k$ such that $Q_k=0$, i.e., $T$ is an infinite simple ring, or $N=0$. If $T$ is an infinite simple ring, then by hypothesise and Theorem \ref{t4}, we deduce that $T$ has infinitely many maximal subrings which is absurd. Thus assume that $N=0$. Therefore $T$ embeds in $\prod_{i\in I}T/M_i$, and similar to the proof of Proposition \ref{t8}$(7)$, we conclude that $T$ is embeddable in a division ring which is impossible. Thus $T$ has infinitely many maximal subrings.
\end{proof}

We need the following main lemma for more conclusions.

\begin{lem}\label{t9}
Let $T$ be an infinite ring and $\bigcap_{i\in I} M_i=0$ where $\{M_i\ |\ i\in I\}\subseteq \Maxl(T)\cap \Max(T)$. Then either $T$ has infinitely many maximal subring or $C(T)$ is integrally closed in $T$. In particular, if $T$ is a (left) quasi-duo $J$-semisimple ring, then either $T$ has infinitely many maximal subring or $C(T)$ is integrally closed in $T$.
\end{lem}
\begin{proof}
Assume that $T$ has only finitely many maximal subrings. Consequently for any $i\in I$, the division ring $T/M_i$ has only finitely many maximal subrings. Hence by Theorem \ref{t1}, for each $i\in I$, the center of $T/M_i$ is integrally closed in $T/M_i$, for otherwise if there exists $i\in I$ such that $T/M_i$ has a non-central element which is integral over the center of $T/M_i$, then $T/M_i$ has infinitely many maximal subrings, which is absurd. Now assume that $\alpha$ is a non-central element of $T$ which is integral over the center of $T$. Thus there exists $\beta\in T$ such that $\alpha\beta-\beta\alpha\neq 0$. Since  $\bigcap_{i\in I} M_i=0$, we infer that there exists $i\in I$ such that $\alpha\beta-\beta\alpha\notin M_i$. Therefore $\alpha+M$ is a non-central element of the division ring $T/M_i$. Clearly, $\alpha+M_i$ is integral over the center of $T/M_i$, which is impossible. Thus, the center of $T$ is integrally closed in $T$. The last part is evident.
\end{proof}

Now we have the following main result.

\begin{thm}\label{t11}
Let $T$ be a $K$-algebra over an infinite field $K$. Then either $T$ has infinitely many maximal subrings (in fact, $|\RgMax(T)|\geq |K|$) or the following hold:
\begin{enumerate}
\item $T$ is a quasi-duo ring. In particular, $\Max(T)=\Maxl(T)=\Maxr(T)=\Prml(T)=\Prmr(T)$.
\item For each $M\in \Max(T)$, the center of $T/M$ is integrally closed in $T/M$.
\item $J(T)=\bigcap_{M\in \Max(T)}M$ and $T/J(T)$ is a reduced ring. In particular, $N(T)\subseteq J(T)$. Moreover, the center of $T/J(T)$ is integrally closed in $T/J(T)$.
\end{enumerate}
\end{thm}
\begin{proof}
If $T$ has a maximal left ideal $M$, which is not an ideal of $T$, then by Corollary \ref{pt2}, we deduce that $|K|\leq |\Maxl(T)\setminus \Max(T)|$. Thus by Proposition \ref{t8}$(1)$, we conclude that $|K|\leq |\RgMax(T)|$. Hence if $|\RgMax(T)|<|K|$ we obtain that each maximal left ideal of $T$ is an ideal of $T$, i.e., $T$ is a left quasi-duo ring, and by an analogous proof, $T$ is a right quasi-duo ring. Consequently $T$ is a quasi-duo ring. Thus $(1)$ holds. $(2)$ is an immediate consequence of Proposition \ref{t1}. For $(3)$, by $(1)$ we infer that $J(T)$ is the intersection of all maximal ideals of $T$ and therefore $T/J(T)$ embeds in the reduce ring $\prod_{M\in \Max(T)} T/M$ (note that $T$ is a quasi-duo ring and therefore $T/M$ is a division ring for each $M\in \Max(T)$). Hence $T/J(T)$ is a reduced ring and thus $N(T)\subseteq J(T)$. The last part of $(3)$ is obvious by Lemma \ref{t9} (note, $T/J(T)$ is an infinite ring, because it contains a copy of the infinite field $K$).
\end{proof}

\begin{cor}\label{t12}
Let $T$ be a prime $K$-algebra over an infinite field $K$ with $J(T)=0$, then either $|\RgMax(T)|\geq |K|$ or $T$ embeds in a division ring; in particular, $T$ is a domain.
\end{cor}
\begin{proof}
Assume that $|\RgMax(T)|<|K|$, then by the proof of the previous theorem $T$ embeds in the reduced ring $R:=\prod_{M\in \Max(T)} T/M$. Thus $T$ is a reduce prime ring. Therefore $T$ is a domain. Since for each maximal ideal $M$ of $T$, $T/M$ is a division ring, we infer that $T$ is a strongly von Neumann regular ring. Hence by \cite[Theorem 1.2.4]{cohndr}, we conclude that $T$ is embeddable in a division ring, as desired.
\end{proof}

If $R$ is a commutative ring and $M$ is a maximal ideal of $R$ such that $\Char(R/M)=0$, then $R/M$ is a field with infinitely many maximal subrings, by \cite[Corollary 1.5]{azomn}. Consequently, $R$ itself has infinitely many maximal subrings. For noncommutative rings we obtain the following analogue.

\begin{cor}\label{t13}
Let $T$ be a ring with only finitely many maximal subrings. Then the following hold:
\begin{enumerate}
\item For each $M\in \Maxl'(T)\cup \Maxr'(T)$, there exists a prime number $p$ such that $pT\subseteq M$.
\item For each $M\in \Max'(T)$, there exists a prime number $p$ such that $pT\subseteq M$.
\item If $P$ is a left/right primitive ideal of $T$, then either $T/P$ is a division ring or $\Char(T/P)$ is a prime number.
\end{enumerate}
\end{cor}
\begin{proof}
$(1)$ Let $M$ be a maximal left/right ideal of $T$ which is not an ideal of $T$. Since $T$ has only finitely many maximal subring, by Proposition \ref{t8}$(2)$, we conclude that $\mathbb{I}(M)/M$ is a finite field. Hence $\Char(\mathbb{I}(M)/M)=p$, where $p$ is a prime number. Thus $p.1\in M$, because $1\in\mathbb{I}(M)$. Therefore $pT\subseteq M$.\\
$(2)$ The proof is similar to $(1)$ and by the use of Proposition \ref{t8}$(4)$.\\
$(3)$ Let $P$ be a left primitive ideal of $T$, which is not a left/right maximal ideal of $T$. Hence there exists a maximal left ideal $M$ of $T$ such that $P=\lann_T(T/M)$. If $M$ is an ideal of $T$, then $M=\lann_T(T/M)=P$, which is absurd. Thus $M$ is not an ideal of $T$ and therefore by $(1)$, we deduce that there exists a prime number $p$ such that $p(T/M)=0$. Hence $p\in P$ and thus $\Char(T/P)=p$.
\end{proof}

By the previous result, if $T$ is a ring of characteristic zero and there exists a maximal left/right ideal $M$ of $T$ that is not an ideal and moreover either $M\cap \mathbb{Z}=0$ or $T/M$ is a torsion-free/faithful $\mathbb{Z}$-module, then $T$ has infinitely many maximal subrings.

\begin{prop}\label{t14}
Let $T$ be an infinite ring with nonzero characteristic. If $\Max_F(T):=\{M\in \Max(T)\ |\ T/M \ \text{is}\ \text{finite}\}$ is infinite, then $T$ has infinitely many maximal subrings.
\end{prop}
\begin{proof}
Since $n$ has only finitely many prime divisors and for each $M\in \Max_F(T)$, $\Char(T/M)$ is a prime divisor of $n$, we conclude that the exists a prime divisor $p$ of $n$ such that $\Max_{p,F}(T):=\{M\in \Max_F(T)\ |\ \Char(T/M)=p\}$ is infinite. Also note that for each $M\in \Max_{p,F}(T)$, the ring $T/M$ is a finite simple ring, hence we deduce that there exist a natural number $m$ and a finite field $E$ such that $T/M\cong \mathbb{M}_m(E)$. Now we have two cases:\\
$(\bf 1)$ There exist distinct $M_1, M_2,\ldots \in \Max_{p,F}(T)$ such that $T/M_i\cong\mathbb{M}_{m_i}(E_i)$, where $m_i>1$ and $E_i$ is a field. Thus in this case for each $i$, $T$ has a maximal subring $R_i$ which contains $M_i$, for $\mathbb{M}_{m_i}(E_i)$ has a maximal subring by Theorem \ref{pt6}$(4)$. Since for $i\neq j$, $M_i+M_j=T$, we deduce that $R_i\neq R_j$. Hence $R_1, R_2,\ldots$ are distinct maximal subrings of $T$, as desired.\\
$(\bf 2)$ Thus assume that there exist infinitely many distinct maximal ideals $M_1, M_2, \ldots\in \Max_{p,F}(T)$ such that $T/M_i=E_i$ is a finite field. Clearly by our assumption $\mathbb{Z}_p\subseteq E_i$, for each $i\geq 1$. Now if for infinitely many $i$, $\mathbb{Z}_p\subsetneq E_i$, then similar to proof of $(1)$ we obtain that $T$ has infinitely many maximal subrings (note that, $E_i$ is a finite field and $\mathbb{Z}_p\subsetneq E_i$, therefore $E_i$ has a maximal subring, see the comments preceding Theorem \ref{pt6}, and hence $T$ has a maximal subring $R_i$ which contains $M_i$). Thus without loss of generality, we may assume that for each $i\geq 1$, we have $\mathbb{Z}_p=E_i$. Let $P_i=M_{2i}$ and $Q_i=M_{2i-1}$, where $i\geq 1$. Then clearly for each $i\geq 1$, we have $T/(P_i\cap Q_i)\cong\mathbb{Z}_p\times\mathbb{Z}_p$. Since $\mathbb{Z}_p\times\mathbb{Z}_p$ has a maximal subring (in fact, $\{(a,a)\ |\ a\in\mathbb{Z}_p\}$ is the only proper and therefore maximal subring of $\mathbb{Z}_p\times\mathbb{Z}_p$, see the comments preceding Theorem \ref{pt6}), we conclude that $T$ has a maximal subring $R_i$ which contains $P_i\cap Q_i$. Clearly, whenever $i\neq j$, then $(P_i\cap Q_i)+(P_j\cap Q_j)=T$. Thus we deduce that for $i\neq j$ we have $R_i\neq R_j$ and therefore $T$ has infinitely many maximal subrings.
\end{proof}

It is not hard to see that by the proof of the previous result, in fact there exists a prime divisor $p$ of $n$ such that $|\Max_{p,F}(T)|=|\Max_F(T)|\leq |\RgMax(T)|$ for some prime number $p$. Now the next result is in order.

\begin{cor}\label{t15}
Let $T$ be a ring. If $|\Max_F(T)|> \aleph_0$, then $T$ has infinitely many maximal subrings. In fact, $|\Max_F(T)|\leq |\RgMax(T)|$.
\end{cor}
\begin{proof}
Since $|\Max_F(T)|> \aleph_0$ and the number of prime numbers is countable, we conclude that there exists a prime number $p$ such that $|\Max_{p,F}(T)|=|\Max_F(T)|$. Now by a similar proof of Proposition \ref{t14}, we can see that $|\Max_{p,F}(T)|\leq |\RgMax(T)|$ and hence we are done.
\end{proof}

Finally we conclude this section by the following result.

\begin{prop}\label{t16}
Let $T$ be a ring and $S_1\ncong S_2$ be simple $T$-modules. If any of the following conditions holds, then $T$ has infinitely many maximal subrings.
\begin{enumerate}
\item If $E_1=\End_T(S_1)$ is a field which is not an absolutely algebraic field.
\item If $E_1=\End_T(S_1)$ and $E_2=\End_T(S_2)$ are infinite isomorphic ring.
\end{enumerate}
\end{prop}
\begin{proof}
First note that since $S_1$ and $S_2$ are simple left $T$-modules, we infer that there exist maximal left ideals $M_1$ and $M_2$, such that $S_i\cong T/M_i$, as left $T$-modules for $i=1,2$.\\
$(\bf 1)$ Note that $E_1^{op}\cong \mathbb{I}(M_1)/M_1$. Now we have two cases, either $M_1$ is an ideal of $T$ or is not. If $M_1$ is an ideal of $T$, then $T/M_1\cong E_1^{op}$ as rings and therefore $T/M_1$ has infinitely many maximal subrings, by \cite[Corollary 1.5]{azomn}. Hence assume that $M_1$ is not an ideal of $T$. Since $E_1$ is not an absolutely algebraic field, we deduce that $E_1$ is an infinite ring. Consequently $\mathbb{I}(M_1)/M_1$ is an infinite ring. Thus by Theorem \ref{t1}, we obtain that $[M_1]$ and a fortiori $\Maxl(T)\setminus \Max(T)$ are infinite. It follows that from Proposition \ref{t8}$(1)$, $T$ has infinitely many maximal subrings, as desired.\\
$(\bf 2)$ We have two cases, if $M_1$ or $M_2$ is not an ideal of $T$. Then analogous to last part of the proof of $(1)$, we infer that $T$ has infinitely many maximal subrings. Thus suppose that $M_1$ and $M_2$ are ideals of $T$. In this case, we have the following ring isomorphisms
$$T/M_1=\mathbb{I}(M_1)/M_1\cong E_1^{op}\cong E_2^{op}\cong \mathbb{I}(M_2)/M_2=T/M_2.$$
Consequently $T/(M_1\cap M_2)\cong E_1^{op}\times E_1^{op}$, as rings (note $M_1\neq M_2$, because $S_1\ncong S_2$). Since $E_1$ and $E_2$ are infinite division rings, we deduce that $T/(M_1\cap M_2)$ has infinitely many maximal subrings, by Theorem \ref{t5}. Therefore $T$ has infinitely many maximal subrings, which complete the proof.
\end{proof}

\section{Rings integral over their centers}
In this section we investigate on the number of maximal subrings in a ring $T$ that is integral over its center. First we need some preliminary observations and notation.
Let $T$ be a ring. Then it is clear that $\Maxl(T)\cap \Max(T)=\Maxr(T)\cap \Max(T)=\Maxr(T)\cap \Maxl(T)$, we denote this set by $\Max_{lr}(T)$. Define $J'(T):=\bigcap_{M\in \Max_{lr}(T)} M$. Then $J'(T)$ is an ideal of $T$ and $J(T)\subseteq J'(T)$. In particular, if $T$ is either left or right quasi-duo ring, then $J(T)=J'(T)$. Moreover, if $\Maxl(T)\setminus \Max(T)=\{M_1,\ldots,M_n\}$ is finite (which occur, in particular, when $T$ has only finitely many maximal subrings, by Proposition \ref{t8}$(5)$), then $J(T)=J'(T)\cap M_1\cap\cdots\cap M_n$. Note also that by Lemma \ref{t9}, if $T$ is an infinite ring with $J'(T)=0$, then either $T$ has infinitely many maximal subrings or $C(T)$ is integrally closed in $T$. We now state the following result.

\begin{prop}\label{t17}
Let $T$ be a ring which is integral over its center. Then either $T$ has infinitely many maximal subrings or the following hold:
\begin{enumerate}
\item $T/J'(T)$ is a reduce Hilbert commutative ring and $|T/J'(T)|\leq 2^{2^{\aleph_0}}$. In particular, $N(T)\subseteq J'(T)$.
\item $|\Max_{lr}(T)|\leq 2^{\aleph_0}$. For each $M\in \Max_{lr}(T)$, $T/M$ is an absolutely algebraic field (thus $|T/M|\leq \aleph_0$).
\item For each $M\in \Maxl'(T)\cup \Maxr'(T)$, $J'(T)+M=T$.
\end{enumerate}
\end{prop}
\begin{proof}
$(1)$ Assume that $T$ has only finitely many maximal subrings. Thus for each $M\in \Max_{lr}(T)$, the division ring $T/M$ has only finitely many maximal subrings. Since $T$ is integral over its center, we conclude that for each $M\in \Max_{lr}(T)$, the division ring $T/M$ is integral over its center. Thus by Proposition \ref{t1}, we conclude that $T/M$ is a field. Clearly, $T/J'(T)$ embeds in $\prod_{M\in \Max_{lr}(T)} T/M$, thus $T/J'(T)$ is a commutative reduced ring and a fortiori $N(T)\subseteq J'(T)$. Since $T/J'(T)$ is a commutative ring with only finitely many maximal subrings, by \cite[Corollary 1.9]{azomn}, we deduce that $T/J'(T)$ is a Hilbert ring. By \cite[Corollary 1.5]{azomn}, for each $M\in \Max_{lr}(T)$, the field $T/M$ is an absolutely algebraic field and consequently $|T/M|\leq \aleph_0$. Therefore we deduce that
$$|\frac{T}{J'(T)}|\leq |\prod_{M\in \Max_{lr}(T)} T/M|\leq 2^{2^{\aleph_0}}.$$
$(2)$ Since $T/J'(T)$ is a commutative ring, the first part of $(2)$ is an immediate consequence of \cite[(1) of Proposition 2.8]{azomn}. For the last part is evident by the proof of $(1)$.\\ $(3)$ Let $M$ be a maximal left ideal of $T$ which is not an ideal of $T$. Since $T/J'(T)$ is a commutative ring, we conclude that $J'(T)\nsubseteq M$, at once $M+J'(T)=T$ by maximality of $M$.
\end{proof}

As noted in the proof of the previous proposition, if $R$ is a commutative ring with $|\Max(R)|>2^{\aleph_0}$, then $|\RgMax(R)|\geq 2^{\aleph_0}$, see \cite[Proposition 2.8(1)]{azomn}. For noncommutative rings we obtain the following analogue.

\begin{prop}\label{t18}
Let $T$ be a ring which is integral over its center. If one of the following conditions holds, then $T$ has infinitely many maximal subrings.
\begin{enumerate}
\item $|\Maxl(T)|>2^{\aleph_0}$.
\item $|\Maxr(T)|>2^{\aleph_0}$.
\item $|\Max(T)|>2^{\aleph_0}$.
\end{enumerate}
\end{prop}
\begin{proof}
We prove $(1)$ and $(3)$. The proof of $(2)$ is similar to $(1)$.\\
$(1)$ First note that if $\Maxl(T)\setminus \Max(T)$ is infinite, then by Proposition \ref{t8}$(1)$, we deduce that $T$ has infinitely many maximal subrings. Hence assume that $\Maxl(T)\setminus \Max(T)$ is finite. Thus $|\Maxl(T)\cap \Max(T)|=|\Maxl(T)|>2^{\aleph_0}$. Now since $\Max_{lr}(T)=\Maxl(T)\cap \Max(T)$, we deduce that $T$ has infinitely many maximal subrings by Proposition \ref{t17}$(2)$.\\
$(3)$ If $\Max(T)\setminus \Maxl(T)$ is infinite, then by Proposition \ref{t8}$(5)$, we conclude that $T$ has infinitely many maximal subrings. Hence assume that  $\Max(T)\setminus \Maxl(T)$ is finite and therefore $|\Max(T)\cap \Maxl(T)|=|\Max(T)|>2^{\aleph_0}$. By an analogous proof of $(1)$, we obtain that $T$ has infinitely many maximal subrings.
\end{proof}

As the following main result shows, if a ring $T$ is an algebraic $K$-algebra (and therefore is integral over its center), then either $T$ has infinitely many maximal subrings, or $T$ exhibits algebraic properties similar to commutative rings with only finitely many maximal subrings, see \cite{azomn}.

\begin{thm}\label{t19}
Let $T$ be an algebraic $K$-algebra over an infinite field $K$. Then either $T$ has infinitely many maximal subrings or the following hold:
\begin{enumerate}
\item $T$ is a quasi-duo ring. In particular, $J(T)=\bigcap_{M\in \Max(T)} M$.
\item For each $M\in \Max(T)$, the ring $T/M$ is an absolutely algebraic field which contains a copy of $K$. In particular, $K$ is an absolutely algebraic field and $|K|=|T/M|=\aleph_0$. Moreover, if $D$ is a division ring which is a subring of $T$, then $D$ is an absolutely algebraic field.
\item $|\Max(T)|\leq 2^{\aleph_0}$ and $|T/J(T)|\leq 2^{2^{\aleph_0}}$.
\item $T/J(T)$ is a reduced Hilbert commutative ring. In fact, $T/J(T)$ is a von Neumann regular ring which is integral over its prime subring (i.e., $\mathbb{Z}_p$, where $p=\Char(K)$ is a prime number). Moreover, $T$ is integral over $\mathbb{Z}_p$.
\item For each distinct maximal ideals $M$ and $N$ of $T$, the fields $T/M$ and $T/N$ are not isomorphic.
\item For each proper ideal $I$ of $T$, $N(T/I)=J(T/I)$.
\item Each strongly prime ideal $Q$ of $T$ is a maximal ideal.
\item The contraction of each maximal ideal of $T$ to any subring $R$ of $T$, remains a maximal ideal of $R$. In particular, $J(R)$ is nil.
\end{enumerate}
\end{thm}
\begin{proof}
Assume that $T$ has only finitely many maximal subring. Thus by Theorem \ref{t11}, we deduce that $(1)$ holds. The first part of $(2)$ is clear by Proposition \ref{t17}$(2)$. For the second part of $(2)$, let $M$ be any maximal ideal of $T$. Clearly, $T/M$ contains a copy of $K$ (and similarly $D$), thus $K$ and $D$ are absolutely algebraic fields. Consequently, $\aleph_0\leq |K|\leq |T/M|\leq \aleph_0$ and thus the equalities hold. $(3)$ and the first part of $(4)$ are evident by Proposition \ref{t17}. For the second part of $(4)$, note that $T/J(T)$ is a commutative ring which is integral over the field $K$. Hence we deduce that $T/J(T)$ is a zero-dimensional (reduced) ring and a fortiori, it is a von Neumann ring. Since $T$ has only finitely many maximal subrings, we infer that $T/J(T)$ has only finitely many maximal subring. Thus by \cite[Proposition 2.1]{azomn}, we obtain that $T/J(T)$ is integral over $\mathbb{Z}_p$ (note, $K$ is an absolutely algebraic field). Since $T$ is algebraic over $K$, by \cite[Theorem 4.20]{lam}, $J(T)$ is a nil ideal of $T$. Now let $x\in T$, then there exist $n\in\mathbb{N}$ and  $a_0,\ldots,a_{n-1}\in\mathbb{Z}_p$ such that $(x+J(T))^n+a_{n-1}(x+J(T))^{n-1}+\cdots+a_1(x+J(T))+a_0=0$ in the ring $T/J(T)$. Thus $x^n+a_{n-1}x^{n-1}+\cdots+a_1x+a_0\in J(T)$. From the fact that $J(T)$ is nil, we deduce that there exists a natural number $m$ such that $(x^n+a_{n-1}x^{n-1}+\cdots+a_1x+a_0)^m=0$, i.e., $x$ is integral over $\mathbb{Z}_p$. For $(5)$, assume that $T/M\cong T/N$. Thus $T/(M\cap N)\cong T/M\times T/M$. Now note that $T/M$ contains a copy of $K$, and therefore $T/M$ is an infinite field. Hence by Proposition \ref{t5}, $T/M\times T/M$ has infinitely many maximal subrings. It follows that $T$ has infinitely many maximal subrings, which is absurd. For $(6)$, nothing that $T/I$ contains a copy of $K$ and remains algebraic over $K$, thus we may assume that $I=0$. Because $T/J(T)$ is a reduced ring, it is obvious that $N(T)\subseteq J(T)$. Consequently, as we see in the proof of $(4)$, $J(T)$ is nil and then $J(T)=N(T)$. For $(7)$, note that $T/Q$ has no nonzero nil ideal, hence by $(6)$ we conclude that $J(T/Q)=0$, i.e., $Q$ is an intersection of a family of maximal (left) ideal of $T$. Thus $J(T)\subseteq Q$ and then by $(4)$, we deduce that $T/Q$ is a field, i.e., $Q$ is a maximal (left) ideal of $T$. Finally for $(8)$, let $M$ be a maximal ideal of $T$ and $R$ be a subring of $T$. Since $(R+M)/M$ is a subring of the absolutely algebraic field $T/M$, we obtain that $(R+M)/M$ is a field. It follows that $R/(R\cap M)$ is a field, i.e., $R\cap M\in \Max(R)$. This implies that $J(R)\subseteq J(T)\cap R$ and at once $J(R)$ is a nil ideal (note, $J(T)$ is nil by $(6)$).
\end{proof}

\begin{lem}\label{t20}
Let $T$ be a ring which is integral over its center. Assume that $T$ has only finitely many maximal subrings. If $P$ is a prime ideal of $T$ with $J(T/P)=0$, then either $P$ is a maximal ideal of $T$ or $T/P$ is an integral domain. In fact, either $T/P\cong\mathbb{M}_n(F)$ for some natural number $n>1$ and finite field $F$ or $T/P$ is an integral domain.
\end{lem}
\begin{proof}
Since $T$ has only finitely many maximal subrings, then by Proposition \ref{t8}$(1)$ we deduce that $\Maxl'(T)$ is finite. Hence by $J(T/P)=0$ we infer that there exist $M_1,\ldots, M_n\in \Maxl'(T)$, $n\geq 0$ and $N=\bigcap_{i\in I}M_i$, where $M_i\in \Max(R)\cap \Maxl(T)$ such that $P=M_1\cap\cdots\cap M_n\cap N$. Thus by the comments preceding Proposition \ref{t8}, either $P=M_1\cap\cdots\cap M_n$ or $P=N$, because $P$ is a prime ideal of $T$. If $P=M_1\cap\cdots\cap M_n$, then $T/P$ is a left $T$-submodule of the semisimple $T$-module $T/M_1\times\cdots\times T/M_n$. Consequently $T/P$ is a simple Artinian ring. Hence $P$ is a maximal ideal of $T$. Thus assume that $P=N$. Since $T$ is integral over its center, we infer that for each $i\in I$, the division ring $T/M_i$ is integral over its center. Hence, if for some $i\in I$, $T/M_i$ is not a field, then by Proposition \ref{t1}, we obtain that $T/M_i$ and therefore $T$ have infinitely many maximal subrings which is impossible. Thus for each $i\in I$, $T/M_i$ is a field. Clearly $T/P$ embeds in $\prod_{i\in I}T/M_i$ and therefore $T/P$ is a commutative (reduced prime) ring. It follows that $T/P$ is an integral domain, because $P$ is prime. For the last part, if $T/P$ is not an integral domain, then since $T/P$ is a simple Artinian ring, we infer that $T/P\cong\mathbb{M}_n(D)$, where $D$ is a division ring and $n$ is a natural number. Now we have two cases: If $n=1$, then $T/P$ is a division ring which is integral over its center, hence by Proposition \ref{t1}, we deduce that $T/P$ is a field, for $T/P$ has only finitely many maximal subrings. Thus $T/P$ is an integral domain which is absurd. Now suppose that $n>1$, we claim that $D$ is finite. Otherwise, by Theorem \ref{t4}, we deduce that $T/P$ has infinitely many maximal subrings which is absurd. Thus $D$ is a finite field, as desired.
\end{proof}

\begin{prop}\label{t21}
Let $T$ be a ring which is integral over its center. Assume that $T$ has only finitely many maximal subrings. Then the following hold:
\begin{enumerate}
\item $\Prml(T)=\Max(T)=\Prmr(T)$.
\item If $C(R)$ is a zero-dimensional ring, then an ideal $P$ of $T$ is a maximal ideal of $T$ if and only if $P$ is a strongly prime ideal of $T$.
\end{enumerate}
\end{prop}
\begin{proof}
$(1)$ Let $P$ be a left primitive ideal of $T$. Clearly, $J(T/P)=0$ and therefore by Lemma \ref{t20}, we conclude that either $P$ is a maximal ideal of $T$ or $T/P$ is a commutative left primitive ring. Since a commutative primitive ring is a field, we infer that $T/P$ is a field and consequently $P$ is a maximal ideal of $T$. Similarly each right primitive ideal of $T$ is a maximal ideal.\\
$(2)$ It is clear that each maximal ideal $P$ of $T$ is a strongly prime ideal of $T$, for $T/P$ is a simple ring and therefore has no nonzero nil ideals. Conversely, let $P$ be a strongly prime ideal of $T$, i.e., $T/P$ has no nonzero nil ideal. It is not hard to see that $P\cap C(T)$ is a prime ideal of $C(T)$. It follows that by the assumption $P\cap C(T)$ is a maximal ideal of $C(T)$. This implies that $C(T)/(P\cap C(T))\cong (C(T)+P)/P$ is a field. Since $T$ is integral over its center, we conclude that $T/P$ is integral over the field $E:=(C(T)+P)/P$. Hence by \cite[Theorem 4.20]{lam}, $J(T/P)$ is a nil ideal and therefore $J(T/P)=0$, because $P$ is a strongly prime ideal of $T$. Consequently by Lemma \ref{t20}, either $P$ is a maximal ideal of $T$ or $T/P$ is an integral domain. In the second case note that since $T/P$ is integral over the field $E$, we obtain that $P$ is maximal.
\end{proof}

The next main result in this section shows that if $T$ is a ring that is integral over its center and has only finitely many maximal subrings, then $T/J(T)$ is embeds into a ring of the form $R\times S$, where $R$ is a direct product of certain absolutely algebraic fields and $S$ is a finite semisimple ring.

\begin{thm}\label{t22}
Let $T$ be a ring which is integral over its center. Assume that $T$ has only finitely many maximal subrings. Then for each maximal ideal $M$ of $T$ either $T/M$ is an absolutely algebraic field with only finitely many maximal subrings or $T/M\cong\mathbb{M}_n(F)$ for some finite field $F$ and $n>1$. In particular, if $J(T)=0$, then $T$ embeds into
$$\prod_{i\in I}E_i\times\prod_{j\in J}\mathbb{M}_{n_j}(F_j)$$
where each $E_i$ is an absolutely algebraic field with only finitely many maximal subrings, $n_j>1$ and each $F_j$ is a finite field. Moreover $J$ is finite, $|I|\leq 2^{\aleph_0}$ and if $E_i$ is infinite for some $i\in I$, then for each $i'\in I\setminus \{i\}$, $E_{i}\ncong E_{i'}$, as rings.
\end{thm}
\begin{proof}
Let $M$ be a maximal ideal of $T$, then by the assumption $T/M$ has only finitely many maximal subrings and is integral over its center. Hence by Theorem \ref{t1}, we infer that whenever $T/M$ is a division ring, then $T/M$ is a field with only finitely many maximal subrings and therefore $T/M$ is an absolutely algebraic field, by \cite[Corollary 1.5]{azomn}. Thus assume that $T/M$ is not a division ring. Now by Theorem \ref{t4}, whenever $T/M$ is an infinite ring, then $T/M$ has infinitely many maximal subrings which is impossible. Thus $T/M$ is a finite ring. Therefore $T/M\cong\mathbb{M}_n(F)$, where $F$ is a finite ring and $n$ is a natural number. Since $T/M$ is not a division ring in this case, we deduce that $n>1$. Also note that since $T$ has only finitely many maximal subrings, then $T$ has only finitely many distinct maximal ideals $M$ for them the ring $T/M$ is not a field, by Proposition \ref{t8}$(5)$, say $M_1,\ldots, M_k$ (and suppose that $T/M_i\cong\mathbb{M}_{n_i}(F_i)$, for some finite field $F_i$ and $n_i>1$). For the next part, by Proposition \ref{t8}$(8)$, we have $0=J(T)=\bigcap_{M\in \Max(T)}M$. Consequently by the first part of the proof of the theorem, $0=J(T)=M_1\cap\cdots\cap M_k\cap\bigcap_{i\in I}N_i$, where for each $i\in I$, $E_i:=T/N_i$ is an absolutely algebraic field with only finitely many maximal subrings. Clearly, we may assume that for $i\neq j$ in $I$, $N_i\neq N_j$. Thus $T$ embeds in $T/M_1\times\cdots\times T/M_k\times\prod_{i\in I}T/N_i$. It follows that from Proposition \ref{t18}, $|\Max(T)|\leq 2^{\aleph_0}$ and then $|I|\leq 2^{\aleph_0}$. Finally we show that if $i\neq j$ in $I$ and $T/N_i$ is infinite, then $T/N_i\ncong T/N_j$. Note that if $T/N_i\cong T/N_j$, then by Proposition \ref{t5}, we deduce that $T/(N_i\cap N_j)\cong T/N_i\times T/N_i$ has infinitely many maximal subrings which is impossible. Hence we are done (note that clearly $|J|=k$ is finite).
\end{proof}

In the next section we precisely determine when a ring of the form $T=\prod_{i\in I}\mathbb{M}_{n_i}(E_i)$, where each $E_i$ is a field and $n_i\in\mathbb{N}$, has only finitely many maximal subrings. Below we derive some related results for Artinian/Noetherian rings. First we require the following remark.

\begin{rem}
It is not hard to see that if $S$ is a commutative ring with nonzero characteristic, say $n$, $\Max(S)=\{M_1,\ldots, M_k\}$ and $J(S)$ is a nil ideal of $S$ (in particular, if $S$ is a commutative Artinian ring with nonzero characteristic $n$), then $S$ is integral over $\mathbb{Z}_n$ if and only if each $S/M_i$ is an absolutely algebraic field.
\end{rem}

\begin{thm}\label{t23}
Let $T$ be a ring which is integral over its center $C:=C(T)$. If any of the following condition holds, then $T$ has infinitely many maximal subrings.
\begin{enumerate}
\item $C$ is an Artinian ring which is either uncountable or of zero characteristic and or is not integral over its prime subring.
\item $C$ is a Noetherian ring with $|C|>2^{\aleph_0}$.
\end{enumerate}
\end{thm}
\begin{proof}
$(1)$ First note that by \cite[Propositions 1.4 and 2.4]{azkra} and the previous remark, $C$ has a maximal ideal $M$ such that $C/M$ is a field which is either uncountable or of zero characteristic and or is not absolutely algebraic field, respectively. Hence in any cases, $K:=C/M$ is not an absolutely algebraic field. By \cite[Proposition 1.2]{blair}, $T$ has a prime ideal $Q$ such that $Q\cap C=M$. Clearly, $K=C/M\subseteq C(T/Q)$ and $T/Q$ is integral over $K$ and therefore $T/Q$ is integral over its center. Thus by Theorem \ref{t19}$(2)$, we conclude that $T/Q$ has infinitely many maximal subrings. Hence $T$ has infinitely many maximal subrings.\\
$(2)$ First assume that $C$ is an integral domain. By the proof of \cite[Corollary 2.7]{azkrc}, for each maximal ideal $M$ of $C$, there exists a natural number $n$, such that $C/M^n$ is an uncountable ring. Since $C/M^n$ is a zero dimensional noetherian local ring with only maximal ideal $M/M^n$, by the first part of the proof of $(1)$, we deduce that $C/M$ is an uncountable field and similar to the proof of $(1)$, we conclude that $T$ has infinitely many maximal subrings. If $C$ is not an integral domain, then by the proof of \cite[Theorem 2.9]{azkrc}, $C$ has a minimal prime ideal $P$ such that $|C/P|=|C|$. By \cite[Proposition 1.2]{blair}, $T$ has a prime ideal $Q$ such that $Q\cap C=P$. Clearly $T/Q$ is integral over $C/P\subseteq C(T/Q)$ (and therefore $T/Q$ is integral over its center). Similar to the first part of the proof of $(2)$, for each maximal ideal $M/P$ of $C/P$, the field $(C/P)/(M/P)$ is uncountable. Since $T/Q$ is integral over $C/P$, by \cite[Proposition 1.2]{blair}, $T/Q$ has a prime ideal $N/Q$ such that $(N/Q)\cap (C/P)=M/P$ and clearly $(T/Q)/(N/Q)$ is integral over $K:=(C/P)/(M/P)\subseteq C((T/Q)/(N/Q))$ (and therefore $(T/Q)/(N/Q)$ is integral over its center). Now since $(T/Q)/(N/Q)$ contains an uncountable field $K=(C/P)/(M/P)$, we immediately conclude that $(T/Q)/(N/Q)$ has infinitely many maximal subrings by Theorem \ref{t19}$(2)$. Thus $T$ has infinitely many maximal subrings too.
\end{proof}

\begin{thm}\label{t24}
Let $T$ be a left Artinian algebraic $K$-algebra over an infinite field $K$. Then either $T$ has infinitely many maximal subrings or $T$ is countable and is integral over its prime subring.
\end{thm}
\begin{proof}
Assume that $T$ has only finitely many maximal subrings, then by Theorem \ref{t19}$(4)$, we infer that $T/J(T)$ is a von Neumann regular commutative ring and $T$ is integral over $\mathbb{Z}_p$, where $p=\Char(K)$ is a prime number. Now note that there exists a natural number $m$ such that $J(T)^m=0$, for $T$ is a left Artinian ring. Since $T/J(T)$ is an Artinian commutative ring with only finitely many maximal subrings, by \cite[Theorem 2.5$(2)$]{azomn}, we deduce that $T/J(T)$ is a countable ring. Hence by Lemma \ref{pt7}, $T$ is a countable ring.
\end{proof}

\begin{cor}\label{t25}
Let $T$ be a left Noetherian algebraic $K$-algebra over an infinite field $K$. Then either $T$ has infinitely many maximal subrings or $T$ is a countable left Artinian ring which is integral over its prime subring.
\end{cor}
\begin{proof}
First note that since $T$ is an algebraic $K$-algebra over a field $K$, we infer that $J(T)$ is a nil ideal of $T$, by \cite[Theorem 4.20]{lam}. Therefore by \cite[Theorem 10.30]{lam}, $J(T)$ is a nilpotent ideal of $T$, because $T$ is a left Noetherian ring. Hence $T$ is a left Artinian ring, by \cite[Theorem 4.15]{lam}. Thus by Theorem \ref{t24}, we deduce the assertion.
\end{proof}

Finally we conclude this section by the following result for infinite direct product.

\begin{thm}\label{t26}
Let $\{T_i\}_{i\in I}$ be an infinite family of rings and $T=\prod_{i\in I} T_i$. If each of the following conditions holds, then $T$ has infinitely many maximal subrings:
\begin{enumerate}
\item If each $T_i$ is integral over its center.
\item If $T$ is integral over its center.
\end{enumerate}
\end{thm}
\begin{proof}
$(1)$ Let $J$ be the set of all $i\in I$ such that $T_i$ is not a quasi-duo ring. Now we have two cases. First assume that $J$ is infinite. Without loss of generality we may suppose that for each $j\in J$, the ring $T_j$ is not a left quasi-duo ring. Consequently, for each $j\in J$, let $M_i$ be a maximal left ideal of $T_j$ which is not an ideal of $T_j$, and $N_j=\prod_{i\in I}A_i$, where $A_i=T_i$ for $i\neq j$ and $A_j=M_j$. Clearly, each $N_j$ is a maximal left ideal of $T$ which is not an ideal of $T$. Hence $T$ has infinitely many maximal left ideal which are not ideals of $T$. Thus $T$ has infinitely many maximal subrings by Proposition \ref{t8}$(1)$. Now, assume that $J$ is finite. Without loss of generality we may suppose that $J=\emptyset$, i.e., each $T_i$ is a quasi-duo ring with only finitely many maximal subrings. Therefore by Proposition \ref{t17}, for each $i\in I$, the ring $T_i/J(T_i)$ is a commutative ring. Take $J:=\prod_{i\in I} J(T_i)$, then clearly $T/J\cong \prod_{i\in I} T/J(T_i)$. Thus $T/J$ is a product of an infinite family of commutative rings. Consequently, by \cite[Remark 3.18]{azkrc} (or see \cite[Theorem 3.9]{azomn}), $T/J$ has infinitely many maximal subrings and so does $T$. Hence $(1)$ holds.\\
$(2)$ It suffices to show that each $T_i$ is integral over its center and use $(1)$. Clearly, $C_T(T)=\prod_{i\in I}C_{T_i}(T_i)$. Let $k\in I$ and $t\in T_k$, and put $x=(x_i)_{i\in I}$ where $x_i=0$ for $i\neq k$ and $x_k=t$. By our assumption, $x$ is integral over the center of $T$. Hence there exist $n\in\mathbb{N}$ and $c_0, c_1,\ldots, c_{n-1}$ in $C_T(T)$ such that $x^n+c_{n-1}x^{n-1}+\cdots+c_1x+c_0=0$. Since $C_T(T)=\prod_{i\in I}C_{T_i}(T_i)$, we have $c_j=(c_{ji})_{i\in I}$, for $0\leq j\leq n-1$, where $c_{ji}\in C_{T_i}(T_i)$. Now by calculating the $k$-th component of $x^n+c_{n-1}x^{n-1}+\cdots+c_1x+c_0=0$, we obtain that $t^n+c_{n-1k}t^{n-1}+\cdots+c_{1k}t+c_{0k}=0$. It follows that $t$ is integral over $C_{T_i}(T_i)$ and we are done by $(1)$.
\end{proof}

\section{Direct product and matrix rings}
In this section we study the number of maximal subrings in matrix rings and in certain direct products of rings. In \cite[Theorem 3.9]{azomn}, the authors precisely determined when a direct product of a family of commutative rings has only finitely many maximal subrings. In particular, if $T$ is an infinite direct product of a family of commutative rings $\{R_i\}_{i\in I}$, then $|\RgMax(T)|\geq 2^{|I|}$. As shown in the final result of the previous section, if $T=\prod_{i\in I} T_i$ is a direct product of an infinite family of rings, and either each $T_i$ is integral over its center, or $T$ itself is integral over its center, then $T$ has infinitely many maximal subrings. This immediately implies that whenever each $T_i=\mathbb{M}_{n_i}(E_i)$, with $E_i$ is a field and $n_i\in\mathbb{N}$, then $T$ has infinitely many maximal subrings (note that each $\mathbb{M}_{n_i}(E_i)$ is finite-dimensional over its center and therefore integral over its center). In the next proposition we give precise characterization of when such a product has only finitely many maximal subrings.

\begin{prop}\label{t27}
Let $T=\prod_{i\in I}\mathbb{M}_{n_i}(E_i)$, where each $E_i$ is a field and $n_i\in\mathbb{N}$. Then $T$ has only finitely many maximal subrings if and only if the following hold:
\begin{enumerate}
\item $I$ is finite.
\item For each $i\in I$, if $n_i>1$, then $E_i$ is finite.
\item For each $i\neq j$ in $I$, if $n_i=n_j=1$, and $E_i$ is infinite, then $E_i\ncong E_j$, as rings.
\item For each $i\in I$ such that $n_i=1$, $E_i$ has only finitely many maximal subrings.
\end{enumerate}
Consequently, if $T$ has only finitely many maximal subrings, then $T=E_1\times\cdots\times E_r\times \mathbb{M}_{n_1}(E_{r+1})\times\cdots\mathbb{M}_{n_s}(E_{r+s})$, where $n_{i}>1$ and $r,s\geq 0$. In particular, $T$ is a semisimple ring. Moreover, if $A=E_1\times\cdots\times E_r$ and $B=\mathbb{M}_{n_1}(E_{r+1})\times\cdots\mathbb{M}_{n_s}(E_{r+s})$, then each maximal subring of $T$ is of the form $A\times B_1$, where $B_1$ is a maximal subring of $B$, or $A_1\times B$, where $A_1$ is a maximal subring of $A$.
\end{prop}
\begin{proof}
First assume that $T$ has only finitely many maximal subrings. This immediately implies that each $\mathbb{M}_{n_i}(E_i)$ has finitely many maximal subrings. Hence if $n_i=1$, $E_i$ has only finitely many maximal subrings and if $n_i>1$, then by Theorem \ref{t4}, $\mathbb{M}_{n_i}(E_i)$ is finite and at once $E_i$ is finite too. Therefore $(2)$ and $(4)$ hold. Although the proof of $(1)$ is obvious by the comment preceding the proposition, but we give another proof. To see this, let $J:=\{i\in I\ |\ n_i=1\}$, $J'=I\setminus J$, $A=\prod_{i\in J} E_i$ and $B=\prod_{i\in J'} \mathbb{M}_{n_i}(E_i)$. Then clearly $T=A\times B$. Consequently by our assumption $A$ and $B$ has only finitely many maximal subrings. Thus by \cite[Theorem 3.9]{azomn}, we obtain that $J$ is finite. It is clear that by $(5)$ of Proposition \ref{t8}, $J'$ is finite and therefore $I$ is finite. Finally note that $(3)$ is an immediate consequence of Proposition \ref{t5}. Conversely, suppose that $(1)-(4)$ hold and let $A$ and $B$ are defined as above. We prove that $T$ has only finitely many maximal subrings. To see this first note that $A$ and $B$ are homomorphically non-isomorphic rings, hence by Lemma \ref{pt8}, each maximal subring of $T$ is of the form $A\times B_1$, where $B_1$ is a maximal subring of $B$, or $A_1\times B$, where $A_1$ is a maximal subring of $A$. This implies that $T$ has only finitely many maximal subrings, for $B$ is finite and by \cite[Corollary 3.12 and of Proposition 3.15$(1)$]{azomn}, $A$ has only finitely many maximal subrings. Thus we are done.
\end{proof}

\begin{prop}\label{t28}
Let $T=\prod_{i\in I}\mathbb{M}_{n_i}(D_i)$, where each $D_i$ is a division ring, $1<n_i\in\mathbb{N}$. Then $T$ has only finitely many maximal subrings if and only if $T$ is a finite ring.
\end{prop}
\begin{proof}
Assume that $T$ has only finitely many maximal subrings. Then analogous to the proof of the previous theorem we deduce that $I$ is finite and for each $i\in I$, $\mathbb{M}_{n_i}(D_i)$ has only finitely many maximal subrings. Hence by Theorem \ref{t4}, $D_i$ is finite and thus $T$ is finite. The converse is evident.
\end{proof}

\begin{cor}\label{t29}
Let $T=\prod_{i\in I}\mathbb{M}_{n_i}(D_i)$, where each $D_i$ is a division ring which is algebraic over its center (in particular, if each $D_i$ is finite dimensional over its center). Then $T$ has only finitely many maximal subrings if and only if the following hold:
\begin{enumerate}
\item $I$ is finite.
\item For each $i\in I$ whenever $n_i>1$, then $D_i$ is a finite field.
\item For each $i\in I$ whenever $n_i=1$, then $D_i$ is a field with only finitely many maximal subrings.
\item For each $i\neq j$ in $I$, if $n_i=n_j=1$ and $D_i$ is infinite, then $D_i\ncong D_j$, as rings.
\end{enumerate}
\end{cor}
\begin{proof}
First assume that $T$ has only finitely many maximal subrings. Thus for each $i\in I$, we conclude that $\mathbb{M}_{n_i}(D_i)$ has only finitely many maximal subrings. Now if $n_i>1$, then by Theorem \ref{t4} we deduce that $D_i$ is a finite field. Hence assume that $n_i=1$, we claim that $D_i$ is a field. Otherwise by Proposition \ref{t1}, $D_i$ has infinitely many maximal subrings, which is impossible. Thus for each $i\in I$, we conclude that $D_i$ is a field and hence we are done by Proposition \ref{t27}.
\end{proof}

In \cite[Corollary 3.10]{azomn}, the authors proved that for a commutative ring $R$, the ring $R\times R$ has only finitely many maximal subrings if and only if $R$ is a semilocal ring (i.e., has only finitely many maximal ideals) with only finitely many maximal subrings (in particular, $R$ is integral over $\mathbb{Z}_n$ for some $n>1$) and for ever maximal ideal $M$ of $R$, the residue field $R/M$ is finite. For noncommutative rings we obtain the following generalization.

\begin{prop}\label{t30}
Let $R$ be a ring and $T=R\times R$. If $T$ has only finitely many maximal subrings, then the following hold:
\begin{enumerate}
\item $\Max(R)$, $\Maxl(R)$ and $\Maxr(R)$ are finite sets. In particular, $\Max(T)$, $\Maxl(T)$ and $\Maxr(T)$ are finite sets.
\item For each $M\in \Max(R)$, $R/M$ is finite. In particular, for each maximal ideal $Q$ of $T$, the ring $T/Q$ is finite.
\item $R/J(R)$ is a finite semisimple ring. In particular, $T/J(T)$ is a finite semisimple ring.
\item For each maximal one-sided ideal $M$ of $R$, $R/M$ is finite (i.e., each simple left/right $R$-module is finite). In particular, for each maximal one-sided ideal $N$ of $T$, $T/N$ is finite (i.e., each simple left/right $T$-module is finite).
\end{enumerate}
\end{prop}
\begin{proof}
First note that for each maximal ideal $M$ of $R$, $T/(M\times M)\cong R/M\times R/M$ has only finitely many maximal subrings. Thus by Proposition \ref{t5}, $R/M$ is finite. Hence the first part of $(2)$ holds. Also note that for each $M\in \Max(R)$, $T$ has a maximal subring $R_M$ which contains $M\times M$ (by the comments preceding Theorem \ref{pt6}) and clearly whenever $M\neq N\in \Max(R)$, then $R_M\neq R_N$, because $(M\times M)+(N\times N)=T$. Hence we conclude that $\Max(R)$ is finite. Since $T$ has only finitely many maximal subrings, we infer that $R$ has only finitely many maximal subring. This implies that $\Maxl(R)\setminus \Max(R)$ is finite, by $(1)$ of Proposition \ref{t8}. Consequently $\Maxl(R)$ is finite, because $\Max(R)$ is finite by the first part of the proof. Analogously $\Maxr(R)$ is finite. For the second part of $(1)$, note that maximal (left/right) ideals of $T$ are of the forms $M\times R$ or $R\times M$, where $M$ is a maximal (left/right) ideal of $R$. Hence the conclusions are immediate by the first part of $(1)$. By the latter comment and the first part of $(2)$, we deduce the second part of $(2)$. For $(3)$, note that since $R$ has only finitely many maximal left ideals, we conclude that $R/J(R)$ is a semisimple ring. Hence by $(2)$, we have $R/J(R)$ is finite. The second part of $(3)$ is evident, because $J(T)=J(R)\times J(R)$. $(4)$ is an immediate consequences of $(3)$, because $R/J(R)$ and $T/J(T)$ are finite rings.
\end{proof}

From the previous result we immediately conclude that if $R$ is a ring for which $T=R\times R$ has only finitely many maximal subrings, then there exist finitely many pairwise non-isomorphic finite simple left $R$-modules $S_1,\ldots,S_n$ such that every simple left $R$-module is isomorphic to one of the $S_i$.

\begin{prop}\label{t31}
Let $R$ be a ring and $T=\mathbb{M}_n(R)$ where $n>1$. If $T$ has only finitely many maximal subring, then  the following hold:
\begin{enumerate}
\item $R$ has only finitely many maximal subrings.
\item For each maximal ideal $M$ of $R$, the ring $R/M$ is a finite ring. In particular, for each maximal ideal $Q$ of $T$, the ring $T/Q$ is finite.
\item $\Max(R)$ and therefore $\Max(T)$ are finite. Moreover, $\Maxl(R)$ and $\Maxr(R)$ are finite.
\item $R/J(R)$ and therefore $T/J(T)$ are finite semisimple rings. In particular, $\Maxl(T)$ and $\Maxr(T)$ are finite.
\item For each maximal one-sided ideal $M$ of $R$, $R/M$ is finite (i.e., each simple left/right $R$-module is finite). In particular, for each maximal one-sided ideal $N$ of $T$, $T/N$ is finite (i.e., each simple left/right $T$-module is finite).
\end{enumerate}
\end{prop}
\begin{proof}
$(1)$It is obvious that for each maximal subring $S$ of $R$, $\mathbb{M}_n(S)$ is a maximal subring of $T$, see \cite[Theorem 3.1$(2)$]{azq}. Thus $R$ has only finitely many maximal subrings.\\ $(2)$ Note that for each maximal ideal $M$ of $R$, $A_M=\mathbb{M}_n(M)$ is a maximal ideal of $T$ and $T/A_M\cong \mathbb{M}_n(R/M)$. By the assumption that $T$ has only finitely many maximal subrings, we deduce that $T/A_M$ has only finitely many maximal subrings too. Hence by Theorem \ref{t4}, we deduce that $R/M$ is finite (note, $n>1$ and therefore $T/A_M$ is not a division ring). The second part of $(2)$ is obvious by the first part of $(2)$, because each maximal ideal of $T$ is of the form $A_M$, for some maximal ideal $M$ of $R$.\\
$(3)$ With the notation of $(2)$, since $n>1$, by Theorem \ref{pt6}$(4)$, we conclude that $\mathbb{M}_n(R/M)$ has a maximal subring. Hence, $T$ has a maximal subring which contains $A_M$, say $T_M$. In particular, if $M\neq N\in \Max(R)$, then $T_M\neq T_N$, because $A_M+A_N=T$. Thus $\Max(R)$ is finite, a fortiori $\Max(T)$ is finite. For the second part of $(3)$, note that by $(1)$, $R$ has only finitely many maximal subrings, hence Proposition \ref{t8}$(1)$, implies that $\Maxl(R)\setminus \Max(R)$ and $\Maxr(R)\setminus \Max(R)$ are finite. Therefore $\Maxl(R)$ and $\Maxr(R)$ are finite, because $\Max(R)$ is finite by the first part of $(3)$.\\
$(4)$ Since $\Maxl(R)$ is finite, we infer that $R/J(R)$ is a semisimple ring. Nothing that for each maximal ideal $M$ of $R$, $R/M$ is a finite ring, consequently $R/J(R)$ is a finite ring. Now from the fact that $J(T)=\mathbb{M}_n(J(R))$, we deduce that $T/J(T)$ is finite semisimple too. Therefore $\Maxl(T)$ and $\Maxr(T)$ are finite.\\
$(5)$ It is obvious by $(4)$.
\end{proof}

We now state several corollaries.

\begin{cor}\label{t32}
Let $R$ be an infinite ring with $J(R)$ is nilpotent (in particular, if $R$ is an infinite left/right Artinian ring). Then $\mathbb{M}_n(R)$, where $n>1$, and $R\times R$ have infinitely many maximal subrings.
\end{cor}
\begin{proof}
Assume that $\mathbb{M}_n(R)$ or $R\times R$ has only finitely many maximal subrings. Then by Propositions \ref{t30} or \ref{t31}, we deduce that $R/J(R)$ is a finite ring. Therefore by Lemma \ref{pt7}, we conclude that $R$ is a finite ring, which is impossible. Hence we are done.
\end{proof}

\begin{cor}\label{t33}
Let $R$ be a zero-dimensional ring with $\Char(R)=0$. Then $\mathbb{M}_n(R)$, where $n>1$, and $R\times R$ have infinitely many maximal subrings.
\end{cor}
\begin{proof}
Assume that $\mathbb{M}_n(R)$ or $R\times R$ has only finitely many maximal subrings. Then by Propositions \ref{t30} or \ref{t31}, we deduce that for each maximal ideal $M$ of $R$, $R/M$ is a finite ring and therefore $\Char(R/M)\neq 0$. Now note that there exists a prime ideal $P$ of $R$ such that $P\cap\mathbb{Z}=0$. Since $R$ is a zero-dimensional ring we deduce that $P$ is a maximal ideal of $R$ and therefore $R/P$ is finite. But from $\mathbb{Z}\cap P=0$, we deduce that $R/P$ contains a copy of $\mathbb{Z}$, which is absurd.
\end{proof}

\begin{cor}\label{t34}
Let $R$ be a $K$-algebra over an infinite field $K$. Then $\mathbb{M}_n(R)$, where $n>1$, and $R\times R$ has infinitely many maximal subrings.
\end{cor}
\begin{proof}
Assume that $\mathbb{M}_n(R)$ or $R\times R$ has only finitely many maximal subrings. Then by Propositions \ref{t30} or \ref{t31}, we deduce that $R/J(R)$ is a finite ring. But clearly $R/J(R)$ contains a copy of $K$, which is impossible for $K$ is infinite. Hence we are done.
\end{proof}

\begin{prop}\label{t35}
Let $T$ be a semilocal ring (in particular, if $T$ is a left/right Artinian ring) with only finitely many maximal subrings. Then $T/J(T)\cong D_1\times\cdots\times D_n\times \mathbb{M}_{n_1}(E_1)\times\cdots\times \mathbb{M}_{n_k}(E_k)$, where each $D_i$ is a division ring, $1\leq i\leq n$, $E_j$ is a finite field, $1\leq j\leq k$ and $n_j>1$, for some $n,k\geq 0$. Moreover, whenever $D_i$ is infinite, then $D_i\ncong D_j$ as rings, for each $j\neq i$. 
\end{prop}
\begin{proof}
Since $T/J(T)$ is a semisimple ring, we deduce that $T/J(T)\cong \prod_{i\in I}\mathbb{M}_{n_i}(D_i)$, where $I$ is finite, $D_i$ is a division ring and $n_i\in\mathbb{N}$. Clearly $T$ has only finitely many maximal ideals. Also note that since $T$ has only finitely many maximal subrings, then by Proposition \ref{t8}$(1)$, we conclude that $\Maxl(T)\setminus \Max(T)$ is finite. Therefore $\Maxl(T)$ is finite. Similarly $\Maxr(R)$ is finite. Obviously for each $i\in I$, the ring $\mathbb{M}_{n_i}(D_i)$ has only finitely many maximal subrings, thus by Corollary \ref{pt4}, we deduce that whenever $i\in I$ and $n_i>1$, then $D_i$ is finite, i.e., $D_i$ is a finite field. Consequently $T/J(T)\cong D_1\times\cdots\times D_n\times \mathbb{M}_{n_1}(E_1)\times\cdots\times \mathbb{M}_{n_k}(E_k)$, where each $D_i$ is a division ring, $1\leq i\leq n$ and $E_j$ is a finite field, $1\leq j\leq k$ and $n_j>1$, for some $n,k\geq 0$. Finally if for some $1\leq i\neq j\leq n$, $D_i\cong D_j$ as rings, then $T$ has an ideal $I$ such that $T/I\cong D_i\times D_j\cong D_i\times D_i$. Since $T$ has only finitely many maximal subrings, we deduce that $T/I$ has only finitely many maximal subrings. Thus by Proposition \ref{t5}, $D_i$ is finite, which is impossible. Hence we are done.
\end{proof}

We conclude this paper by the following immediate corollary.

\begin{cor}\label{t36}
Let $T$ be a semilocal ring. If $T$ has only finitely many maximal subrings, then $T$ has only finitely many maximal left/right ideals.
\end{cor}

\vspace{0.5cm}
\centerline{\Large{\bf Acknowledgement}}
The author is grateful to the Research Council of Shahid Chamran University of Ahvaz (Ahvaz-Iran) for
financial support (Grant Number: SCU.MM1404.721)
\vspace{0.5cm}


\end{document}